\documentclass[12pt]{amsart}
\usepackage{tikz-cd}
\usetikzlibrary{matrix,arrows,decorations.pathmorphing}
\usepackage{amssymb}
\usepackage{amsmath,amscd}
\usepackage{mathrsfs}
\usepackage{url}
\usepackage{cite}
\usepackage{fullpage}
\usepackage{hyperref}
\usepackage{verbatim}
\usepackage{comment}
\usepackage{fancyvrb}
\usepackage{fvextra}
\usepackage{mathrsfs}
\newtheorem{thm}{Theorem}[section]
\newtheorem{prop}[thm]{Proposition}
\newtheorem{lem}[thm]{Lemma}

\usepackage{wasysym}
\usepackage{graphicx}
\usepackage{wrapfig}

\theoremstyle{definition}
\newtheorem{definition}[thm]{Definition}
\newtheorem{example}[thm]{Example}
\newtheorem{rem}[thm]{Remark}

\numberwithin{equation}{section}
\newcommand{\nf}{\mathrm{nf}}
\newcommand{\aV}{\mathcal{U}}
\newcommand{\U}{\mathcal{U}}

\newcommand{\aW}{\mathcal{X}}

\newcommand{\M}{\mathcal{M}}
\newcommand{\X}{\mathcal{X}}

\newcommand{\B}{\mathcal{B}}

\newcommand{\zz}{\mathbb{Z}}
\newcommand{\qq}{\mathbb{Q}}

\newcommand{\C}{\mathcal{C}}
\newcommand{\Cp}{\mathcal{C}}
\newcommand{\Pp}{\mathcal{P}}

\newcommand{\p}{\mathbb{P}}
\newcommand{\pp}{\mathbb{P}}

\renewcommand{\H}{\mathcal{H}}
\newcommand{\Hp}{\mathcal{H}}

\newcommand{\F}{\mathcal{F}}
\renewcommand{\P}{\mathcal{P}}
\newcommand{\E}{\mathcal{E}}
\renewcommand{\O}{\mathcal{O}}

\newcommand{\Mg}{\M_g}
\newcommand{\Ub}{\mathcal{V}}

\DeclareMathOperator{\rank}{rank}
\DeclareMathOperator{\trun}{Trun}

\DeclareMathOperator{\ch}{ch}

\DeclareMathOperator{\Aut}{Aut}
\DeclareMathOperator{\GL}{GL}
\DeclareMathOperator{\SL}{SL}

\DeclareMathOperator{\Supp}{Supp}

\DeclareMathOperator{\coker}{coker}

\DeclareMathOperator{\Sym}{Sym}

\DeclareMathOperator{\PGL}{PGL}
\DeclareMathOperator{\BSL}{BSL}

\DeclareMathOperator{\BGL}{BGL}

\DeclareMathOperator{\Proj}{Proj}

\newcommand{\hannah}[1]{{\color{teal} ($\spadesuit$ Hannah: #1)}}

\begin{document}
\title{Tautological classes on Low-degree Hurwitz Spaces}

\author{Samir Canning, Hannah Larson}
\thanks{During the preparation of this article, S.C. was partially supported by NSF RTG grant DMS-1502651. H.L. was supported by the Hertz Foundation and NSF GRFP under grant DGE-1656518. This work will be part of S.C.'s and H.L.'s Ph.D. theses.}
\email{srcannin@ucsd.edu}
\email{hlarson@stanford.edu}
\maketitle
\begin{abstract}
Let $\H_{k,g}$ be the Hurwitz stack parametrizing degree $k$, genus $g$ covers of $\pp^1$. 
We define the tautological ring of $\H_{k,g}$ and we show that all Chow classes, except possibly those supported on the locus of ``factoring covers," are tautological up to codimension roughly $g/k$ when $k \leq 5$.
The set-up developed here is also used in our subsequent work \cite{part2}, wherein we prove new results about the structure of the Chow ring for $k \leq 5$.
\end{abstract}

\begin{comment}
\section*{To dos}
\begin{itemize}
    \item What should we do with script and cal letters? Should we distinguish here to make the picard paper easier or relegate all of that distinguishing to the picard paper? \hannah{I think all script letters are gone}
    \item be careful about sharp vs not sharp inequalities on the bounds. 
    \item I think we should make $\aW_{r,d}'$ be the vector bundle over $\B_{r,d}'$. This just means have to use some other letter in the open embedding lemmas. \hannah{I think this is all fixed, but keep an eye out for any that I missed}
    \item references in the intro
\end{itemize}
\end{comment}

\section{Introduction}
When studying intersection theory of moduli spaces, one often introduces certain natural or ``tautological" classes coming from the universal family. In the words of Mumford \cite{Mum}:

\begin{quote}
\vspace{.1in}
\textit{Whenever  a  variety  or  topological  space  is  defined  by  some  universal property, one expects that by  virtue of its defining property, it possesses certain  cohomology  classes  called  tautological  classes.}
\end{quote}

\vspace{.1in}
\noindent
For example, in the case of  the moduli space of curves $\M_g$, the tautological classes Mumford proposes to study are the \emph{kappa classes}, defined as follows. Let $f: \Cp \to \M_g$ be the universal curve; then $\kappa_i := f_*(c_1(\omega_f)^{i+1}) \in A^i(\M_g)$, the Chow ring of $\Mg$. The \emph{tautological ring}, denoted $R^*(\M_g) \subseteq A^*(\M_g)$, is the subring of the rational Chow ring generated by the kappa classes.

In this paper, we study the intersection theory of the Hurwitz space $\Hp_{k,g}$, the moduli space of degree $k$, genus $g$ covers of $\pp^1$, up to automorphisms of the target. Following Mumford's philosophy, let us begin by introducing a notion of tautological classes. Let $\Cp$ be the universal curve and $\Pp$ the universal $\pp^1$-fibration over the Hurwitz space $\Hp_{k,g}$:
\begin{center}
\begin{tikzcd}
\Cp \arrow{r}{\alpha} \arrow{rd}[swap]{f} & \Pp \arrow{d}{\pi} \\
& \Hp_{k,g}.
\end{tikzcd}
\end{center}
We define the tautological subring of the Hurwitz space $R^*(\Hp_{k,g}) \subseteq A^*(\Hp_{k,g})$ to be the subring generated by classes of the form $f_*(c_1(\omega_f)^i \cdot \alpha^*c_1(\omega_\pi)^j) = \pi_*(\alpha_*(c_1(\omega_f)^i)  \cdot c_1(\omega_\pi)^j)$.

In general, determining the full Chow ring of a moduli space --- such as $\M_g$ or $\Hp_{k,g}$ --- may be quite difficult. Having established a notion of tautological classes, however, it makes sense to split the study of the intersection theory of a moduli space into two parts:
\begin{itemize}
\item \textbf{Question 1:} To what extent are classes tautological? If they exist, what can we say about the support of non-tautological classes?
\vspace{.1in}
\item \textbf{Question 2:} What is the structure of the tautological ring? Although the full Chow ring may be complicated, one hopes that the tautological ring has a more easily described structure.
\end{itemize}
In this paper, we provide an answer to Question 1 for $\H_{k,g}$ with $k \leq 5$. The ground work we develop here will also be important for addressing Question 2, which we undertake in subsequent work \cite{part2}.

Before stating our results, we highlight some known results about the Chow ring of $\M_g$ related to Question 1 for context.
\begin{enumerate}
    \item[(1a)] (codimension 1) Codimension 1 classes are tautological: $A^1(\M_g) = R^1(\M_g)$ \cite{harerpic}.
    \item[(1b)] (low genus) For $g \leq 6$, all classes are tautological: $A^*(\M_g) = R^*(\M_g)$ \cite{Mum, F2, F3, Iz, PV}.
    \item[(1c)] (bielliptics) In genus $12$, the fundamental class of the bielliptic locus $B_{12}$ is \emph{not} tautological: $[ B_{12}] \notin R^*(\M_{12})$ \cite{VZ}. 
%    \item[(1d)]  (stable cohomology) Mumford conjectured, and Madsen--Weiss later proved, that the stable cohomology of $\M_g$ is tautological. \hannah{do you like that? on the other hand, it may be better to keep things just about Chow because then should we be saying some other things about cohomology? Or perhaps we could just mention this at the later point I referenced it?}
\end{enumerate}
%\sam{We could talk about Mumford conjecture here and how our result is kind of like that} \hannah{What is Mumford's conjecture? Is it that the stable cohomology is tautological?}
%\sam{yes. }
\begin{rem}
Building upon the results for Hurwitz spaces in this paper and its sequel \cite{part2}, we extend (1b) to prove $A^*(\M_g) = R^*(\M_g)$ for all $g \leq 9$ in \cite{CL}.
\end{rem}

\noindent
Meanwhile, for the Hurwitz space $\H_{k,g}$, the previously known results regarding tautological classes are as follows:

\begin{enumerate}
    \item[(1a)] (codimension 1) Codimension 1 classes are tautological $A^1(\H_{k,g}) = R^1(\H_{k,g})$ for $k \leq 5$ \cite{DP} and $k > g-1$ \cite{Mu}. The general case remains an open conjecture known as the \emph{Picard rank conjecture}.
    \item[(1b)] (low degree) For $k \leq 3$, all classes are tautological: $A^*(\H_{k,g}) = R^*(\H_{k,g})$. In the case $k = 2$, it is well-known that $A^*(\H_{2,g}) = \qq$; the case $k =3$ is due to Patel--Vakil \cite{PV2}.
\end{enumerate}

Our main theorems make significant progress towards answering Question 1 for the next open cases: the Hurwitz spaces $\H_{4,g}$ and $\H_{5,g}$. 
A degree $4$ cover $C \to \pp^1$ can factor as two degree two covers $C \to C' \to \pp^1$. Let $\H_{4,g}^{\mathrm{nf}} \subset \H_{4,g}$ denote the open locus of \emph{non-factoring covers}, or equivalently covers whose monodromy group is not contained in the dihedral group $D_4$. By $R^*(\H_{4,g}^{\mathrm{nf}})$ we mean the image of $R^*(\H_{4,g})$ under the restriction map $A^*(\H_{4,g}) \to A^*(\H_{4,g}^{\mathrm{nf}})$.

\begin{thm} \label{thm4}
If they exist, any non-tautological classes on $\H_{4,g}$ are supported on the locus of factoring covers or have codimension at least $(g+3)/4 - 4$. In other words,
\[A^i(\H_{4,g}^{\mathrm{nf}}) = R^i(\H_{4,g}^{\mathrm{nf}})  \quad \text{for all } i < (g+3)/4 - 4.\]
\end{thm}

\begin{rem}The fact that there may be non-tautological classes on the locus of factoring covers should be compared with (1c). In fact, using van Zelm's result that $[B_{12}]$ is not tautological, we establish in \cite[Remark 1.10]{part2} that $\H_{4,12}$ indeed possesses non-tautological classes supported on the factoring locus. \end{rem}

In degree $5$, covers cannot factor, and we obtain the following result.
\begin{thm} \label{thm5}
 If they exist, any non-tautological classes on $\H_{5,g}$ have codimension at least $(g+4)/5 - 16$. In other words,
 \[A^i(\H_{5,g}) = R^i(\H_{5,g}) \quad \text{for all } i < (g+4)/5 - 16.\]
\end{thm}

Theorems \ref{thm4} and \ref{thm5} are reminiscent of the Madsen--Weiss theorem \cite{MW}, which proves Mumford's conjecture that the stable cohomology of $\Mg$ is a polynomial ring in the kappa classes. Edidin \cite[Question 3.34]{edidin} asked if the analogue of the Madsen--Weiss theorem holds in the Chow ring $A^*(\Mg)$, but very little is known about this question. We view Theorems \ref{thm4} and \ref{thm5} as providing some evidence toward a positive answer to Edidin's question.  
Unlike the case of the stable cohomology of $\Mg$, we will show in \cite{part2} that there are many interesting relations among the tautological classes on $\H_{k,g}$ when $3\leq k \leq 5$.

%In subsequent work \cite{part2}, we will show that the tautological Chow groups in the ranges considered above are non-trivial. Thus --- taking $g$ large enough --- Theorem \ref{thm5} provides examples of non-trivial Chow groups of arbitrarily high codimension which are provably all tautological.

\subsection*{Sketch of the proof}
There are three key ingredients to proving Theorems \ref{thm4} and \ref{thm5}. The set up we develop will also be essential for later results determining structure of the tautological ring in \cite{part2}. We shall therefore carry them out in the case $k = 3$ as well, which fits into the same framework.

\vspace{.1in}
 \textit{(1) Useful generators:} 
We first explain how structure theorems of Casnati--Ekedahl for finite covers give rise to a collection of classes on $\H_{k,g}$, which we term Casnati--Ekedahl (CE) classes. In Theorem \ref{CEgen}, we show that all CE classes are tautological and that they generate the tautological ring.
An interesting consequence of this is that, for fixed $k, i$, $\dim R^i(\H_{k,g})$ is bounded above, independent of $g$. (This part works for any $k$; see Remark \ref{not-growing}).

\vspace{.1in}    
\textit{(2) The good open:} For $k = 3, 4, 5$, we define a ``good open" $\H_{k,g}' \subseteq \H_{k,g}$. Using our interpretation of the Casnati--Ekedahl structure theorems, we show that this ``good open" possesses an open embedding inside a vector bundle $\X_{k,g}'$ over a moduli space $\B_{k,g}'$ of pairs of vector bundles on $\pp^1$. The pullbacks of classes along $A^*(\B_{k,g}') = A^*(\X_{k,g}') \to A^*(\H_{k,g}')$ are CE classes (essentially by the definition of CE classes). It follows that $A^*(\H_{k,g}')$ is generated by tautological classes. %\hannah{How do you feel about calling $\H_{k,g}'$ the ``good open" and $\H_{k,g}^\circ$ the ``jet-good open"? Helpful or weird? Even though $\H_{k,g}^\circ$ is not logically needed, it may make sense to introduce it in this paper (and help save space for Crelle!) We can probably get slightly improved codimension bounds on $\H_{k,g}'$.}

\vspace{.1in}
 \textit{(3) Codimension bounds:} By excision, there is a surjection $A^*(\H_{k,g}) \to A^*(\H_{k,g}')$ whose kernel is generated by classes supported on the complement of $\H_{k,g}'$. Thus, the final step is to bound the codimension of the complement of $\H_{k,g}'$. When $k = 3$, it turns out $\H_{k,g}' = \H_{k,g}$, so we recover the result of Patel--Vakil \cite{PV2} that all classes are tautological. When $k = 4$, the complement of $\H_{k,g}'$ contains the locus of covers that factor through a double cover of a low-genus curve. Thus, the complement has codimension 2. However, it turns out that the \emph{non-factoring} covers in the complement of $\H_{k,g}'$ have codimension at least $(g+3)/4 - 4$. This leads to the proof of Theorem \ref{thm4} at the end of Section \ref{co4}.
 Finally, for $k=5$, there are no factoring covers, and we show that the complement of $\H_{k,g}'$ has codimension at least $(g+4)/5 - 16$. With this, we conclude the proof of Theorem \ref{thm5} at the end of Section \ref{op5sec}.

\subsection*{Acknowledgments} We are grateful to our advisors, Elham Izadi and Ravi Vakil, respectively, for the many helpful conversations. We are grateful to Aaron Landesman for his comments and insights.

\section{Notation and conventions} \label{conventions}
We will work over an algebraically closed field of characteristic $0$ or characteristic $p>5$. All schemes in this paper will be taken over this fixed field.
\subsection{Projective bundles} \label{pandg}
We follow the subspace convention for projective bundles: given a scheme (or stack) $X$ and a vector bundle $E$ of rank $r$ on $X$, we set
\[
\p E:=\Proj(\Sym ^{\bullet} E^{\vee}),
\]
so we have the tautological inclusion
\[
\mathcal{O}_{\p E}(-1)\hookrightarrow \gamma^*E,
\]
where $\gamma :\p E\rightarrow X$ is the structure map. Set $\zeta:=c_1(\O_{\p E}(1))$. With this convention, the Chow ring of $\p E$ is given by
\begin{equation} \label{pbt}
A^*(\p E)=A^*(X)[\zeta]/\langle \zeta^r + \zeta^{r-1} c_1(E) + \ldots + c_r(E)\rangle.
\end{equation}
We call this the \emph{projective bundle theorem}.
Note that $1, \zeta, \zeta^2, \ldots, \zeta^{r-1}$ form a basis for $A^*(\pp E)$ as an $A^*(X)$-module. Since
\[ \gamma_* \zeta^i = \begin{cases} 0 & \text{if $i \leq r-2$} \\ 1 & \text{if $i = r-1$,} \end{cases}\]
this determines the $\gamma_*$ of all classes from $\p E$.

\subsection{(Equivariant) Intersection Theory} 
%All of our Chow rings will be taken with rational coefficients. 
Let $X$ be a scheme and suppose $Z \subseteq X$ is a closed subscheme of codimension $c$ and $U$ is its open complement. We denote the Chow ring of $X$ with rational coefficients by $A^*(X)$. The \emph{excision property} of Chow is the right exact sequence
\[A^{*-c}(Z) \rightarrow A^*(X) \rightarrow A^*(U) \rightarrow 0.\]
If one knows $A^*(X)$, then to find the Chow ring of an open $U \subset X$, one must describe the image of $A^{*-c}(Z) \rightarrow A^*(X)$. If $\widetilde{Z} \to Z$ is proper and surjective, then pushforward $A_*(\widetilde{Z}) \to A_*(Z)$ is surjective, see \cite[Lemma 1.2]{V}.
Given a graded ring $R = \bigoplus R^{i}$, let 
\[\trun^dR:= R/\oplus_{i \geq d} R^d\]denote the degree $d$ trunction. With this notation, if the complement of $U \subseteq X$ has codimension $c$, then the excision property implies
\begin{equation} \label{barnot}
\trun^cA^*(X) \xrightarrow{\sim} \trun^c A^*(U).
\end{equation}

Chow rings also satisfy the \emph{homotopy property}: if $V \to X$ is a vector bundle, then the pullback map $A^*(X) \to A^*(V)$ is an isomorphism. This property motivates the definition of equivariant Chow groups as developed by Edidin-Graham in \cite{EG}. Again, we will be using rational coefficients for our equivariant Chow rings.
Let $V$ be a representation of $G$ and suppose $G$ acts freely on $U \subset V$ and the codimension of $V \smallsetminus U$ is greater than $c$.
If $X$ is a smooth scheme and $G$ is a linear algebraic group acting on $X$, Edidin and Graham defined
\[A^c_G(X) := A^c((X \times U)/G),\]
and showed that the graded ring $A^*_G(X)$ possesses an intersection product.
%\hannah{maybe add sentence: The Chow ring of a Deligne--Mumford stack was defined in [Vistoli 1989]. But there's also some definition in Edidin-Graham p. 29. But they're all }
For quotient stacks, one has $A^*([X/G]) \cong A^*_{G}(X)$ by \cite[Proposition 19]{EG}, which may suffice as the definition of the Chow rings of all stacks appearing in this paper.  

%The codimension $1$ part with \emph{integral coefficients} $A^1_{G}(X)$ is isomorphic to Mumford's functorial Picard group of the stack $[X/G]$ by \cite[Proposition 18]{EG}. In particular, the Picard group of $BG$ is isomorphic to the character group of $G$. To avoid confusion about which coefficients we are considering, we will use the notation $A^*(Y)$ to refer to the Chow ring with rational coefficients and the notation $\Pic(Y)$ to refer to the Picard group (with integral coefficients) of a stack $Y$.

\par By Edidin-Graham \cite[Proposition 5]{EG}, there is also an excision sequence for equivariant Chow groups. Let $Z\subseteq X$ be a $G$-invariant closed subscheme of codimension $c$ and $U$ its complement. Then there is an exact sequence
\[
A^{*-c}_{G}(Z)\rightarrow A^*_{G}(X)\rightarrow A^*_{G}(U)\rightarrow 0.
\]
The following lemma is a useful consequence of the excision sequence. See also \cite[Theorem 2]{V2} for a much more general statement.
\begin{lem}\label{Gmbundles}
Suppose $P\rightarrow X$ is a principal $\mathbb{G}_m$-bundle. Then $A^*(P)=A^*(X)/\langle c_1(L) \rangle$, where $L$ is the corresponding line bundle.
\end{lem}
\begin{proof}
By the correspondence between principal $\mathbb{G}_m$-bundles and line bundles over $X$, $P$ is the complement of the zero section of the line bundle $L\rightarrow X$. The excision sequence gives
\[
A^{*-1}(X)\rightarrow A^*(L)\rightarrow A^*(P)\rightarrow 0.
\]
Under the identification of $A^*(L)$ with $A^*(X)$, the first map in the above exact sequence is multiplication by $c_1(L)$, from which the result follows.
\end{proof}

\subsection{The Hurwitz space} \label{hssec}
Given a scheme $S$, an $S$ point of the \emph{parametrized Hurwitz scheme} $\H_{k,g}^\dagger$ is the data of a finite, flat map $C \to \pp^1 \times S$, of constant degree $k$ so that the composition $C \to \pp^1 \times S \to S$ is smooth with geometrically connected fibers. (We do not impose the condition that a cover $C \to \pp^1$ be simply branched.)

The \emph{unparametrized Hurwitz stack} is the $\PGL_2$ quotient of the parametrized Hurwitz scheme. There is also a natural action of $\SL_2$ on $\H_{k,g}^\dagger$ (via $\SL_2 \subset \GL_2 \to \PGL_2$). The natural map $[\H_{k,g}^\dagger/\SL_2] \to [\H_{k,g}^\dagger/\PGL_2]$ is a $\mu_2$ banded gerbe.  It is a general fact that \emph{with rational coefficients}, the pullback map along a gerbe banded by a finite group is an isomorphism \cite[Section 2.3]{PV}. In particular, since we work with rational coefficients throughout, $A^*([\H_{k,g}^\dagger/\PGL_2]) \cong A^*([\H^\dagger_{k,g}/\SL_2])$. It thus suffices to prove all statements for the $\SL_2$ quotient $[\H^\dagger_{k,g}/\SL_2]$, which we shall denote by $\H_{k,g}$ from now on.

Explicitly, the $\SL_2$ quotient $\H_{k,g}$ is the stack whose objects over a scheme $S$ are families $(C\rightarrow P\rightarrow S)$ where $P = \pp V \rightarrow S$ is the projectivization of a rank $2$ vector bundle $V$ with trivial determinant, $C\rightarrow P$ is a finite, flat, finitely presented morphism of constant degree $k$, and the composition $C\rightarrow S$ has smooth fibers of genus $g$.
The benefit of working with $\H_{k,g}$ is that the $\SL_2$ quotient is equipped with a universal $\pp^1$-bundle $\P \to \H_{k, g}$ that has a relative degree one line bundle $\O_{\P}(1)$ (a $\pp^1$-fibration does not).
Working with this $\pp^1$-bundle simplifies our intersection theory calculations.

\section{The Casnati--Ekedahl structure theorem} \label{CEsec}

The main objective of this section is to give a description of stacks of low-degree covers using structure theorems of Casnati--Ekedahl. The descriptions in Sections \ref{tripsec}--\ref{pentsec} are likely well-known but have not previously been spelled out in the language of stacks except in the degree 3 case \cite{BV}, as we shall need them. On a first pass, the reader may wish to skip forward to Section \ref{CEclass}, where we introduce natural classes coming from these structure theorems and prove that they generate the tautological ring.

Generalizing earlier results of Schreyer \cite{S} and Miranda \cite{M}, Casnati--Ekedahl \cite{CE} proved a general structure theorem for degree $k$, Gorenstein covers of integral schemes. Given a degree $k$ cover $\alpha:X\rightarrow Y$ where $Y$ is integral, one obtains an exact sequence
\begin{equation} \label{edef}
0\rightarrow \mathcal{O}_Y\rightarrow \alpha_*\mathcal{O}_X\rightarrow E_\alpha^{\vee}\rightarrow 0,
\end{equation}
where $E_\alpha$ is a vector bundle of rank $k-1$ on $Y$. 
When $\alpha$ is Gorenstein, $\alpha_* \O_X \cong (\alpha_* \omega_{\alpha})^\vee$ by Serre duality.
 Pulling back and using adjunction, we therefore obtain a map 
\begin{equation} \label{com}
\omega_{\alpha}^\vee \to (\alpha^*\alpha_* \omega_{\alpha})^\vee \rightarrow \alpha^*E_\alpha^{\vee},
\end{equation}
which induces a map $X\rightarrow \p E^{\vee}$ that factors $\alpha:X\rightarrow Y$.
\begin{example}[Covers of $\pp^1$] \label{p1ex}
If $\alpha: C \to \pp^1$ is a degree $k$, genus $g$ cover, then we have
\[\deg(E_\alpha^\vee) = \deg(\alpha_*\O_C) = \chi(\alpha_*\O_C) - k = \chi(\O_C) - k = 1 - g - k,\]
so $\deg(E_\alpha) = g+k-1$. The map $C \to \pp E_\alpha^\vee$ factors the canonical embedding $C \hookrightarrow \pp^{g-1}$, where the map $\pp E^\vee_\alpha \to \pp^{g-1}$ is given by the line bundle $\O_{\pp E_\alpha^\vee}(1) \otimes \omega_{\pp^1}$.
Each linear space in the image of $\pp E_\alpha^\vee \to \pp^{g-1}$ is the span of the image of the corresponding fiber of $C \to \pp^1$. 
\end{example}
The Casnati--Ekedahl structure theorem below gives a resolution of the ideal sheaf of $X$ inside of $\p E^{\vee}_\alpha$ \cite{CE}; see also \cite{CN}.
\begin{thm}[Casnati--Ekedahl, Theorem 2.1 of \cite{CE}]\label{CEstructure}
Let $X$ and $Y$ be schemes, $Y$ integral and let $\alpha:X\rightarrow Y$ be a Gorenstein cover of degree $k\geq 3$. There exists a unique $\p^{k-2}$-bundle $\gamma:\p\rightarrow Y$ and an embedding $i:X\hookrightarrow \p$ such that $\alpha=\gamma\circ i$ and $X_y:=\alpha^{-1}(y)\subset \gamma^{-1}(y)\cong \p^{k-2}$ is a nondegenerate arithmetically Gorenstein subscheme for each $y\in Y$. Moreover, the following properties hold.
\begin{enumerate}
    \item $\p \cong \p E^{\vee}_\alpha$ where $ E^{\vee}_\alpha:=\coker(\mathcal{O}_Y\rightarrow \alpha_*\O_X)$.
    \item The composition $\alpha^* E_\alpha \to \alpha^*\alpha_* \omega_{\alpha} \to \omega_\alpha$ is surjective (dually, \eqref{com} does not drop rank) and the ramification divisor $R$ satisfies $\O_X(R) \cong \omega_\alpha \cong \O_X(1) := i^* \O_{\pp E_\alpha^\vee}(1)$.
    \item There exists an exact sequence of locally free $\mathcal{O}_\p$ sheaves
    \begin{equation}\label{CEseq}
        0\rightarrow \gamma^*F_{k-2}(-k)\rightarrow \gamma^*F_{k-3}(-k+2)\rightarrow \cdots\rightarrow \gamma^* F_{1}(-2)\rightarrow \mathcal{O}_\p\rightarrow \mathcal{O}_X\rightarrow 0.
    \end{equation}
    where $F_i$ is locally free on $Y$.
    The restriction of the exact sequence above to a fiber gives a minimal free resolution of $X_y:=\alpha^{-1}(y)$. This sequence is unique up to unique isomorphism. Moreover the resolution is self-dual, meaning there is a canonical isomorphism $\H om_{\O_{\pp}}(F_i, F_{k-2}) \cong F_{k-2-i}$. The ranks of the $F_i$ are 
    \[ 
    \rank F_i=\frac{i(k-2-i)}{k-1}{{k}\choose{i+1}}.
    \]
    \item If $\pp \cong \pp E'^\vee$, then $E' \cong E$ if and only if $F_{k-2} \cong \det E'$ in the resolution \eqref{CEseq} computed with respect to the polarization $\O_{\pp E'^\vee}(1)$.
\end{enumerate}
\end{thm}

\begin{rem} \label{detiso}
There is a canonical isomorphism $F_{k-2} \cong \det E_\alpha$, which we describe here. Following \cite[p. 446]{CE}, let $A_1$ be the image of $\gamma^*F_1(-2) \to \O_{\pp}$, and
for $2 \leq i \leq k-3$,  let $A_i$ denote the image of $\gamma^*F_i(-i-1) \to \gamma^*F_{i-1}(-i)$. We set $A_{k-2}$ to be $\gamma^*F_{k-2}(-k)$. We have exact sequences
\begin{equation} \label{a1}
0 \rightarrow A_1 \rightarrow \O_{\pp} \rightarrow \O_X \rightarrow 0
\end{equation}
and
\begin{equation} \label{afs} 0 \rightarrow A_{i+1} \rightarrow \gamma^*F_i(-i-1) \rightarrow A_i \rightarrow 0.
\end{equation}
First, we claim that
\[R^j \gamma_* \gamma^*F_i(-i-1) \cong \begin{cases} F_{k-2} \otimes \det E^\vee & \text{if $i = j = k-2$} \\  0 & \text{otherwise.} \end{cases}\]
This is very similar to the calculations of \cite[p. 446]{CE}, but twisted up by one.
To prove the first case above, we note that the dualizing sheaf of $\gamma$ is $\omega_\gamma = (\gamma^* \det E)(-k +1)$, and apply Serre duality for $\gamma$, which is of relative dimension $k - 2$.
The other cases follow from the theorem on cohomology and base change and the well-known cohomology of line bundles on projective space.
Tensoring the exact sequences of \eqref{afs} by $\O_{\pp}(1)$ and pushing forward by $\gamma$, the boundary maps provide us with isomorphisms
\[\gamma_* A_1(1)  \cong R^1\gamma_* A_2(1)  \cong R^2 \gamma_* A_3(1)  \cong \cdots \cong R^{k-1}\gamma_*(\gamma^*F_{k-2}(-k+1)) = 0.\]
Similarly, we have
\[R^1\gamma_* A_1(1) \cong R^2\gamma_* A_2(1) \cong \cdots \cong R^{k-2} \gamma_* (\gamma^*F_{k-2}(-k+1)) \cong F_{k-2} \otimes \det E^\vee.\]
On the other hand, tensoring \eqref{a1} with $\O_{\pp}(1)$ and pushing forward by $\gamma$ we obtain
\[0 \rightarrow E \rightarrow \alpha_* \O_X(1) \rightarrow R^1 \gamma_* A_1(1) \rightarrow 0.\]
Recall that $\O_X(1) \cong \omega_\alpha$, so
dualizing \eqref{edef} we see that the cokernel of the left map is $\O_Y$. By the universal property of cokernel, we obtain an isomorphism \[\O_Y \to R^1\gamma_*A_1(1) \cong F_{k-2} \otimes \det E^\vee,\]
or equivalently, an isomorphism $F_{k-2} \cong \det E$.
\end{rem}

 In the cases $k=3,4,5$, using self-duality, only pullbacks of the bundles $E_\alpha$ and $F_1$ and determinants and tensor products thereof appear in the resolution \eqref{CEseq}. 
 We set $F_{\alpha}:=F_1$. Twisting up \eqref{CEseq} by $\O_{\pp}(2)$ and pushing forward by $\gamma$, we see that
 \[F_{\alpha} = \ker(\Sym^2 E_\alpha \twoheadrightarrow \alpha_* \omega_\alpha^{\otimes 2}).\]
 In these low degrees $k = 3, 4, 5$, there is a special map $\delta_\alpha$ in the resolution \eqref{CEseq} from which one can reconstruct the cover. Furthermore, as we shall explain, it is an \emph{open condition} on a space of global sections of all such maps $\delta$ to define a finite cover. This is what distinguishes $k = 3, 4, 5$ and lies at the core of why our methods work in these low degrees. Below we present an equivalence of categories between the category of degree $k$, Gorenstein covers of a scheme $S$ and a category of certain linear algebraic data on $S$. The main content of this step is to point out the ``essential data" of a cover, which we may remember instead of the entire resolution. For the case of triple covers, this was done by Bolognesi--Vistoli \cite{BV}. We give a slightly different perspective below.

\subsection{The category of triple covers} \label{tripsec}
Let $\mathrm{Trip}(S)$ denote the category of Gorenstein triple covers of a scheme $S$: the objects are Gorenstein triple covers $\alpha: X \to S$ and the arrows are isomorphisms over $S$.
Specializing \eqref{CEseq} to the case $k=3$, associated to a cover $\alpha: X \to S$, one obtains a rank $2$ vector bundle $E_\alpha$ and an exact sequence
\[
0\rightarrow \O_{\p E_\alpha^\vee}(-3)\otimes \gamma^*\det E_\alpha \xrightarrow{\delta_\alpha} \mathcal{O}_{\p E_\alpha^{\vee}} \rightarrow \mathcal{O}_X\rightarrow 0.
\] 
Conversely, from the above sequence, we can recover the cover $\alpha:X\rightarrow S$. Indeed, the map $\delta_\alpha$ is a global section in $H^0(\p E_\alpha^{\vee} ,\O_{\p E^{\vee}_{\alpha}}(3)\otimes\gamma^*\det  E_\alpha^{\vee})$, whose zero locus inside of $\p E_\alpha^{\vee}$ is $X$. 
Meanwhile, given any rank $2$ vector bundle $E$ on $S$, it is an open condition on the space of sections
$H^0(\p E^{\vee},\O_{\p E^{\vee}}(3)\otimes\gamma^*\det  E^{\vee})$
for the vanishing of a section $\delta$ to define a finite triple cover: $\delta$ must not be the zero polynomial on any fiber of $\pp E \to S$. Equivalently, if
\begin{equation} \label{3com}
\Phi: H^0(S,\Sym^3  
E\otimes \det  E^{\vee})\xrightarrow{\sim} H^0(\p E^{\vee},\O_{\p E^{\vee}}(3)\otimes\gamma^*\det  E^{\vee})
\end{equation}
denotes the natural isomorphism, then $V(\delta) \subset \pp E^\vee$ is a Gorenstein triple cover so long as $\Phi^{-1}(\delta)$ is non-vanishing.

This ``essential data" is captured by a category $\mathrm{Trip}'(S)$ we now define. The objects of $\mathrm{Trip}'(S)$ are pairs $(E, \eta)$ where $E$ is a rank $2$ vector bundle and $\eta\in H^0(S,\Sym^3 E \otimes \det E^\vee)$ is non-vanshing on $S$. 
An arrow $(E_1, \eta_1) \to (E_2, \eta_2)$ in $\mathrm{Trip}'(S)$ is an isomorphism $E_1 \to E_2$ that sends $\eta_1$ into $\eta_2$.
There is a functor $\mathrm{Trip}(S) \to \mathrm{Trip}'(S)$ that sends $\alpha: X \to S$ to the pair $(E_\alpha, \Phi^{-1}(\delta_\alpha))$. %\hannah{What if we said it: we get $E_\alpha$. Then we take $\Sym^3 E_\alpha \to \alpha_* \omega_{\alpha}$. It has a kernel $L \to \Sym^3 E_\alpha$ and we also have a canonical isomorphism of $L$ with $\det E_\alpha$ produced by the remark. Using that, obtain $\eta_\alpha := \det E_\alpha \cong L \to \Sym^3 E_\alpha$.} 
There is also a functor $\mathrm{Trip}'(S) \to \mathrm{Trip}(S)$ that sends a pair $(E, \eta)$ to the triple cover $V(\Phi(\eta)) \subset \pp E^\vee \to S$.
The following is essentially a restatement of \cite[Theorem 3.4]{CE}, which was proved earlier by Miranda \cite{M}.

\begin{thm}[Miranda, Casnati--Ekedahl] \label{CE3}
The functors above define an equivalence of categories $\mathrm{Trip}(S) \cong \mathrm{Trip}'(S)$.
\end{thm}

\subsection{The category of quadruple covers} \label{quadsec}
Let $\mathrm{Quad}(S)$ denote the category whose objects are Gorenstein quadruple covers $\alpha: X \to S$ and whose arrows are isomorphisms over $S$.
Associated to a degree $4$ cover $\alpha: X \to S$, there is a rank $3$ vector bundle $E_\alpha$ and a rank $2$ vector bundle $F_\alpha$ and a resolution
\begin{equation}\label{deg4res}
0\rightarrow \gamma^*\det  E_\alpha(-4)\rightarrow \gamma^*F_\alpha(-2)\xrightarrow{\delta_\alpha} \O_{\p E_\alpha^\vee} \rightarrow \O_X\rightarrow 0.
\end{equation}
The section $\delta_\alpha \in H^0(\pp E_{\alpha}^{\vee},\O_{\pp E_{\alpha}^{\vee}}(2) \otimes \gamma^* F^\vee)$ corresponds to a relative pencil of quadrics. The cover 
$X$ can be recovered as the vanishing locus of $\delta_\alpha$. By comparing \eqref{deg4res} with the Koszul resolution of $\delta_\alpha$, 
\begin{equation}\label{koszul}
0\rightarrow \gamma^*\det F_\alpha(-4)\rightarrow \gamma^*F_\alpha(-2)\xrightarrow{\delta_\alpha} \O_{\p E^{\vee}_{\alpha}}\rightarrow \O_X\rightarrow 0,
\end{equation}
the uniqueness of Theorem \ref{CEstructure} (3) induces a distinguished isomorphism $\phi_\alpha: \det F_\alpha \cong \det E_\alpha$ (see \cite[p. 450]{CE}). 

We now define a category $\mathrm{Quad}'(S)$ of the corresponding linear algebraic data of a quadruple cover. Given vector bundles $E, F$ on $S$,
there is a natural isomorphism
\begin{equation} \label{4comp}
\Phi: H^0(S, F^\vee \otimes \Sym^2 E) \xrightarrow{\sim} H^0(\pp E^\vee, \gamma^*F^\vee\otimes \O_{\pp E^{\vee}}(2)).
\end{equation}

\begin{definition}
Let $E$ and $F$ be vector bundles of ranks $3$ and $2$ respectively on $S$. We say that a section $\eta \in H^0(S, F^\vee \otimes \Sym^2 E)$ \emph{has the right codimension} at $s\in S$ if the vanishing locus of $\Phi(\eta)$ restricted to the fiber over $s \in S$ is zero dimensional.
\end{definition}

The objects of $\mathrm{Quad}'(S)$ are tuples $(E, F, \phi, \eta)$ where $E$ and $F$ are vector bundles of ranks $3$ and $2$ respectively, $\phi: \det F \cong \det E$ is an isomorphism and $\eta \in H^0(S, F^\vee \otimes \Sym^2 E)$ has the right codimension at all $s \in S$. An arrow in $\mathrm{Quad}'(S)$ is a pair of isomorphisms $\xi: E_1 \to E_2$, and $\psi: F_1 \to F_2$, such that the following diagrams commute
\begin{center}
\begin{tikzcd}
F_1 \arrow{d}[swap]{\psi} \arrow{r}{\eta_1} &\Sym^2 E_1 \arrow{d}{\Sym^2 \xi} \\
F_2 \arrow{r}{\eta_2} &\Sym^2 E_2
\end{tikzcd}
\hspace{1in}
\begin{tikzcd}
\det F_1 \arrow{r}{\phi_1} \arrow{d}[swap]{\det \psi} & \det E_1 \arrow{d} {\det \xi} \\
\det F_2 \arrow{r}{\phi_2} & \det E_2.
\end{tikzcd}
\end{center}

There is a functor $\mathrm{Quad}(S) \to \mathrm{Quad}'(S)$ that sends $\alpha:X \to S$ to $(E_\alpha, F_\alpha, \phi_\alpha, \eta_\alpha)$ where $\eta_\alpha := \Phi^{-1}(\delta_\alpha)$.
%\hannah{maybe this part is confusing because when we go through CE, we're going up and down and up and down. But: given a cover $X \to S$, we get $E_\alpha = \alpha_* \omega_\alpha$ and we get $F_\alpha = \ker \Sym^2 E_\alpha \to \alpha_*\omega_\alpha^{\otimes 2}$, so like, don't we already have a god-given $\eta_\alpha: F_\alpha \to \ker \Sym^2 E_\alpha$? If you were to change $F_\alpha$ by an automorphism you have to change $\eta_\alpha$ with it...except also, with $F_\alpha$ being a subbundle of $\Sym^2 E_\alpha$, can't we ``see" if the automorphism of $F_\alpha$ moves inside $\Sym^2 E_\alpha$?}
There is also a functor $\mathrm{Quad}'(S) \to \mathrm{Quad}(S)$ that sends a tuple $(E, F, \phi, \eta)$ to the quadruple cover $V(\Phi(\eta)) \subset \pp E^\vee \to S$. The following is essentially a restatement of \cite[Theorem 4.4]{CE}.
\begin{thm}[Casnati--Ekedahl]\label{CE4}
The functors above define an equivalence of categories $\mathrm{Quad}(S) \cong \mathrm{Quad}'(S)$.
\end{thm}
\begin{proof}
Work of Casnati--Ekedahl established that the composition $\mathrm{Quad}(S) \to \mathrm{Quad}'(S) \to \mathrm{Quad}(S)$ is equivalent to the identity, as $V(\delta_\alpha) \to S$ is naturally identified with the cover $\alpha: X \to S$. 

We must provide a natural isomorphism of  $\mathrm{Quad}'(S) \to \mathrm{Quad}(S) \to \mathrm{Quad}'(S)$ with the identity on $\mathrm{Quad}'(S)$.  Suppose we are given $(E, F, \phi, \eta) \in \mathrm{Quad}'(S)$.
We want to define an arrow $(E, F, \phi, \eta) \to (E_\alpha, F_\alpha, \phi_\alpha, \eta_\alpha)$.
Let $X = V(\Phi(\eta)) \subset \pp E^\vee$, and $\alpha: X \to S$. The Kosul resolution of $\Phi(\eta)$ is
\[0 \rightarrow (\gamma^*\det F)(-4) \rightarrow \gamma^*F(-2) \xrightarrow{\Phi(\eta)} \O_{\pp E^\vee} \rightarrow \O_X \rightarrow 0\]
and is exact since $\eta$ has the right codimension at all $s \in S$.
We break this into two sequences
\begin{equation} \label{fh} 0 \rightarrow (\gamma^*\det F)(-4) \rightarrow \gamma^*F(-2) \rightarrow A \rightarrow 0
\end{equation}
and
\begin{equation} \label{sh}
0 \rightarrow A \rightarrow \O_{\pp E^\vee} \rightarrow \O_X \rightarrow 0.
\end{equation}
Pushing forward  \eqref{sh} we get a short exact sequence on $S$:
\[0 \rightarrow \O_S \rightarrow \alpha_*\O_X \rightarrow R^1\gamma_* A \rightarrow 0.\]
Using \eqref{fh}, we obtain isomorphisms $R^1\gamma_*A \cong R^2\gamma_*(\gamma^*\det F)(-4) \cong \det F \otimes R^2\gamma_* \O_{\pp E^\vee}(-4)$. Because the dualizing sheaf of $\gamma$ is $\omega_\gamma = \O_{\pp E^\vee}(-3) \otimes \gamma^*\det E$, using Serre duality, we obtain an isomorphism $R^2\gamma_* \O_{\pp E^\vee}(-4) \cong \det E^\vee \otimes E^\vee$. Now the universal property of cokernel produces an isomorphism
\[E_\alpha^\vee = \coker(\O_S \to \alpha_* \O_X) \xrightarrow{\sim} R^1\gamma_*A \cong  \det F \otimes \det E^\vee \otimes E^\vee.\]
Meanwhile $\phi$ determines an isomorphism $\det F \otimes \det E^\vee \cong \O_S$. Composing with this, and dualizing, we obtain an isomorphism $\xi: E \to E_\alpha$.
%In turn, this induces isomorphisms $\O_{\pp E^\vee}(1)|_X \cong \O_{\pp E^\vee_\alpha}(1)|_X \cong \omega_\alpha$. 
Next, we have a commuting diagram
\begin{center}
\begin{tikzcd}
0 \arrow{r} & F \arrow[dashed]{d}[swap]{\psi} \arrow{r}{\eta} & \Sym^2 E \arrow{r}\arrow{d}{\Sym^2 \xi}& \alpha_* \O_X(2) \arrow{r} \arrow{d} & 0 \\
0 \arrow{r} & F_\alpha \arrow{r}{\eta_\alpha} &\Sym^2 E_\alpha \arrow{r}  & \alpha_* \omega_\alpha^{\otimes 2} \arrow{r} & 0
\end{tikzcd}
\end{center}
where the left vertical map is induced by the universal property of kernel.
Note that for any $t \in \O_S^{\times}(S)$, the diagram 
\begin{equation} \label{ddiag}
\begin{tikzcd}
F \arrow{d}[swap]{t^2 \cdot \psi} \arrow{r}{\eta} & \Sym^2 E \arrow{d}{\Sym^2(t \cdot \xi)}\\
F_\alpha \arrow{r}{\eta_\alpha} &\Sym^2 E_\alpha 
\end{tikzcd}
\end{equation}
also commutes.
Finally, the cover $\alpha$ determines an isomorphism $\phi_\alpha: \det F_\alpha \cong \det E_\alpha$. It may not agree with $\phi$, but since the maps below involve isomorphisms of line bundles, there exists some $t \in \O_S^\times(S)$ such that the following diagram commutes
\begin{center}
\begin{tikzcd}
\det F \arrow{r}{\phi} \arrow{d}[swap]{t\cdot \det \psi} & \det E \arrow{d} {\det \xi} \\
\det F_\alpha \arrow{r}{\phi_\alpha} & \det E_\alpha.
\end{tikzcd}
\end{center}
Since $E$ is rank $3$ and $F$ is rank $2$, this implies the diagram
\begin{equation} \label{d2}
\begin{tikzcd}
\det F \arrow{r}{\phi} \arrow{d}[swap]{\det(t^2 \cdot \psi)} & \det E \arrow{d} {\det(t \cdot \xi)} \\
\det F_\alpha \arrow{r}{\phi_\alpha} & \det E_\alpha
\end{tikzcd}
\end{equation}
also commutes. Thus, the pair of isomorphisms
$t \cdot \xi: E \to E_\alpha$ and $t^2 \cdot \psi: F \to F_\alpha$
determine an arrow $(E, F, \phi, \eta) \to (E_\alpha, F_\alpha, \phi_\alpha, \eta_\alpha)$.
\end{proof}

\subsection{The category of regular pentagonal covers} \label{pentsec}
By the Casnati--Ekedahl theorem, each degree $5$ Gorenstein cover $\alpha: X \to S$ determines a resolution
\begin{equation} \label{res5}
0\rightarrow \gamma^*\det  E_\alpha(-5)\rightarrow \gamma^*(F_\alpha^{\vee}\otimes\det E_\alpha)(-3)\xrightarrow{\delta_\alpha} \gamma^*F_\alpha(-2)\rightarrow \O_\p\rightarrow \O_X\rightarrow 0,
\end{equation}
where $E_\alpha$ has rank $4$ and $F_\alpha$ has rank $5$.
Casnati showed that the map $\delta_\alpha$ is alternating in the sense that it can be identified with a section of $\wedge^2\pi^*F_\alpha \otimes \gamma^*\det  E_\alpha^{\vee}(1)$. For any pair of vector bundles $E$ and $F$, via push-pull, we have an identification
\begin{equation} \label{comp}
\Phi: H^0(S, \H om(E^\vee \otimes \det E, \wedge^2 F)) \xrightarrow{\sim} H^0(\pp E^\vee, \gamma^*(\wedge^2 F \otimes \det E^\vee)(1)).
\end{equation}
Hence, $\delta_\alpha$ corresponds to a map $\eta_\alpha := \Phi^{-1}(\delta_\alpha): E_\alpha^\vee \otimes \det E_\alpha \to \wedge^2 F_\alpha$. Throughout this section we shall write $E' := E^\vee \otimes \det E$.
A degree $5$ cover $\alpha: X \to S$ is called \emph{regular} if $\eta_\alpha$ is injective as a map of vector bundles (i.e. the cokernel of $\eta_\alpha$ is locally free). Casnati notes that if $\alpha^{-1}(s)$ is a local complete intersection scheme for all $s \in S$, then $\alpha$ is regular, so all covers we need will be regular. We let $\mathrm{Pent}(S)$ denote the category whose objects are regular, degree $5$ Gorenstein covers $\alpha : X \to S$ and arrows are isomorphisms over $S$.

Regular degree $5$ covers have a nice geometric description. Indeed, if the cover is regular, then $\eta_\alpha$ corresponds to an injective map $ 
E'_\alpha \rightarrow \wedge^2 F_\alpha$, which induces an embedding
\begin{equation} \label{emb}
\p E_\alpha' \hookrightarrow \p(\wedge^2 F_\alpha).
\end{equation}
Given a section $\delta \in H^0(\pp E^\vee, \gamma^*(\wedge^2 F \otimes \det E^\vee)(1))$, we let $D(\delta) \subset \pp E^\vee$ be the subscheme defined by the vanishing of $4 \times 4$ Pfaffians of $\delta$. When $\alpha$ is regular, we can recover $X = D(\delta_\alpha)$, which is also the same as the scheme defined by the $3 \times 3$ minors of $\delta_\alpha$ (Proposition 3.5 of \cite{C}). These $3 \times 3$ minors are pullbacks to $\pp E_\alpha'$ along \eqref{emb} of the equations that define the Grassmannian bundle $G(2,F_\alpha) \subset \p(\wedge^2 F_\alpha)$ under its relative Pl\"ucker embedding.  Using a resolution of the relative Grassmannian, Casnati obtains another resolution of $\O_X$ in equation (3.5.2) of \cite{C}. 
Comparing this resolution with \eqref{res5},  the uniqueness of Theorem \ref{CEstructure} (2) induces a distinguished isomorphism $\epsilon: F_\alpha \otimes \det F_\alpha^\vee \otimes (\det E_\alpha)^{\otimes 2} \to F_\alpha$ (see p. 467 of \cite{C}). Moreover, both of these vector bundles arise as subbundles of $\Sym^2 E_\alpha$ and the projectivization of $\epsilon$ induces the identity on points (as it must be the restriction of the identity on $\pp(\Sym^2 E_\alpha)$).
Hence, we obtain an isomorphism of line bundles
\[\O_{\pp F_\alpha}(1) \otimes \det F_\alpha^\vee \otimes (\det E_\alpha)^{\otimes 2} \cong  
\O_{\pp(F_\alpha \otimes \det F_\alpha^\vee \otimes \det E_\alpha^2)}(1) \cong \epsilon^*\O_{\pp F_\alpha}(1) = \O_{\pp F_\alpha}(1).
\]
which induces a distinguished isomorphism $\phi_\alpha:  (\det E_\alpha)^{\otimes 2} \cong \det F_\alpha$.

Now we define a category $\mathrm{Pent}'(S)$ that keeps track of the associated linear algebraic data of regular degree $5$ covers.

\begin{definition}
Suppose we are given vector bundles $E$ and $F$ on $S$ of ranks $4$ and $5$. Let $\eta \in H^0(S, \H om(E', \wedge^2 F))$ be a global section. We say
$\eta$ has the \emph{right codimension} if every fiber of $D(\Phi(\eta)) \subset \pp E^\vee \to S$ is $0$-dimensional and $\eta: E' \to \wedge^2 F$ is injective with locally free cokernel.
\end{definition}
We define $\mathrm{Pent}'(S)$ to be the category whose objects are tuples $(E, F, \phi, \eta)$ where $E$ and $F$ are vector bundles on $S$ of ranks $4$ and $5$ respectively, $\phi$ is an isomorphism $(\det E)^{\otimes 2} \cong \det F$ and $\eta \in H^0(S, \H om(E^\vee \otimes \det E, \wedge^2 F))$ has the right codimension.
An arrow $(E_1, F_1, \phi_1, \eta_1) \to (E_2, F_2, \phi_2, \eta_2)$ in $\mathrm{Pent}'(S)$ is pair of isomorphisms $\xi: E_1 \to E_2$ and $\psi: F_1 \to F_2$ such that the following two diagrams commute
\begin{center}
\begin{tikzcd}
E_1'  \arrow{d}[swap]{\det \xi \otimes (\xi^{-1})^\vee} \arrow{r}{\eta_1} & \wedge^2 F_1 \arrow{d}{\wedge^2 \psi} \\
E_2' \arrow{r}{\eta_2} &\wedge^2 F_2
\end{tikzcd}
\hspace{1in}
\begin{tikzcd}
\det E_1^{\otimes2}\arrow{d}[swap]{(\det \xi)^{\otimes2}} \arrow{r}{\phi_1} & \det F_1 \arrow{d}{\det \psi}\\
\det E_2^{\otimes2} \arrow{r}{\phi_2} & \det F_2.
\end{tikzcd}
\end{center}
There is a functor $\mathrm{Pent}(S) \to \mathrm{Pent}'(S)$ that sends $\alpha:X \to S$ to the tuple $(E_\alpha, F_\alpha, \phi_\alpha, \eta_\alpha)$. There is also a functor $\mathrm{Pent}'(S) \to \mathrm{Pent}(S)$ that sends a tuple $(E, F, \phi, \eta)$ to the degree $5$ cover $D(\Phi(\eta)) \subset \pp E^\vee \to  S$. The following is essentially a restatement of \cite[Theorem 3.8]{C}.

\begin{thm}[Casnati] \label{C5}
The above functors define an equivalence of categories $\mathrm{Pent}(S) \cong \mathrm{Pent}'(S)$.
\end{thm}
\begin{proof}
The fact that $\mathrm{Pent}(S) \to \mathrm{Pent}'(S) \to \mathrm{Pent}(S)$ is equivalent to the identity was established by Casnati. We provide further details here that $\mathrm{Pent}'(S) \to \mathrm{Pent}(S) \to \mathrm{Pent}'(S)$ is naturally isomorphic to the identity on $\mathrm{Pent}'(S)$. Let $(E, F, \phi, \eta) \in \mathrm{Pent}(S)$ be given and let $X = D(\Phi(\eta))$ and $\alpha: X \to S$. 
By  (3.5.2) of \cite{C}, $\O_X$ admits a resolution
\begin{align*}
&0 \rightarrow \gamma^*(\det F^{-2} \otimes \det E^5(-5) \rightarrow \gamma^*(F^\vee \otimes \det F^{-1} \otimes \det E^3)(-3) \\
 &\qquad \qquad \qquad \rightarrow \gamma^*(F \otimes \det F^{-1} \otimes \det E^2)(-2) \rightarrow \O_{\pp E^\vee} \rightarrow \O_X \rightarrow 0.
\end{align*}
Let $A_1$ be the image of $\gamma^*(F \otimes \det F^{-1} \otimes \det E^2)(-2) \to \O_{\pp E^\vee}$.
When we push forward the above equation by $\gamma$, we obtain
\[0 \rightarrow \O_S \rightarrow \alpha_* \O_X \rightarrow R^1\gamma_* A_1 \rightarrow 0.\]
We use a similar method as in Remark \ref{detiso} to produce isomorphisms
\[E_\alpha \cong R^1\gamma_* A_1 \cong R^2\gamma_* A_2 \cong R^3\gamma_* (\gamma^*(\det F^{-2} \otimes \det E^5))(-5) \cong \det F^{-2} \det E^4 \otimes E^\vee.\]
Using $\phi$, we turn this into an isomorphism $E_\alpha^\vee \cong E^\vee$, which we dualize to define $\xi: E \cong E_\alpha$.
Using the uniqueness of the CE resolution, we also get an isomorphism $F \otimes \det F^{-1} \otimes \det E^2 \to F_{\alpha}$. Making use of $\phi$ again, we obtain an isomorphism $\psi: F  \cong F_{\alpha}$. This in turn induces a map $G(2, F) \to G(2, F_\alpha)$ which sends $X$ into $X$. Since the points of $X$ span each fiber of $\pp E' \cong \pp E_\alpha'$, the following diagram of linear maps of spaces commutes
\begin{center}
\begin{tikzcd}
\pp E'  \arrow{d} \arrow{r} & \pp(\wedge^2 F) \arrow{d} \\
\pp E_\alpha' \arrow{r} & \pp(\wedge^2 F_\alpha).
\end{tikzcd}
\end{center}
In other words, there exists $t \in \O_S^{\times}(S)$ such that the first diagram below commutes, and, since $E$ has rank $4$, so does the second:
\begin{equation} \label{ew}
\begin{tikzcd}
E'  \arrow{d}[swap]{t \cdot \det \xi \otimes (\xi^{-1})^\vee} \arrow{r}{\eta} & \wedge^2 F \arrow{d}{\wedge^2 \psi} \\
E_\alpha' \arrow{r}{\eta_\alpha} &\wedge^2 F_\alpha
\end{tikzcd}
\hspace{1in}
\begin{tikzcd}
E'  \arrow{d}[swap]{\det(t\cdot \xi) \otimes ((t \cdot \xi)^{-1})^\vee} \arrow{r}{\eta} & \wedge^2 F \arrow{d}{\wedge^2 (t \cdot \psi)} \\
E_\alpha' \arrow{r}{\eta_\alpha} &\wedge^2 F_\alpha.
\end{tikzcd}
\end{equation}

Finally, we must compare $\phi$ and $\phi_\alpha$. Since all the maps involved are isomorphisms of line bundles, there exists some $x \in \O_S^\times(S)$ such that the first diagram below commutes; recalling that $E$ is rank $4$ and $F$ is rank $5$, hence so does the second:
\begin{center}
\begin{tikzcd}
\det E^{\otimes2} \arrow{d}[swap]{x \cdot \det (t \cdot \xi)^2} \arrow{r}{\phi} & \det F \arrow{d}{\det(t \cdot \psi)}\\
\det E_\alpha^{\otimes2} \arrow{r}{\phi_\alpha} & \det F_\alpha
\end{tikzcd}
\hspace{1in}
\begin{tikzcd}
\det E^{\otimes2} \arrow{d}[swap]{\det (x^2t \cdot \xi)^2} \arrow{r}{\phi} & \det F \arrow{d}{\det(x^3 t \cdot \psi)}\\
\det E_\alpha^{\otimes2} \arrow{r}{\phi_\alpha} & \det F_\alpha.
\end{tikzcd}
\end{center}
Finally, note that
\begin{center}
\begin{tikzcd}
E'  \arrow{d}[swap]{\det(x^2t\cdot \xi) \otimes ((x^2t \cdot \xi)^{-1})^\vee} \arrow{r}{\eta} & \wedge^2 F \arrow{d}{\wedge^2 (x^3t \cdot \psi)} \\
E_\alpha' \arrow{r}{\eta_\alpha} &\wedge^2 F_\alpha.
\end{tikzcd}
\end{center}
also commutes, as it just rescales both vertical maps of the second diagram in \eqref{ew} by $x^6$. Hence, pair of isomorphisms $x^2t \cdot \xi: E \to E_\alpha$ and $x^3t \cdot \psi: F \to F_\alpha$ define an arrow $(E, F, \phi, \eta) \to (E_\alpha, F_\alpha, \phi_\alpha, \eta_\alpha)$ in $\mathrm{Pent}'(S)$.
%\hannah{If the left-hand map were $\det \xi \otimes \xi^\vee$ (without the inverse), then rescaling $\xi$ by $x$ would rescale the left map by $x^5$. Then, once the first diagram commutes, there's no rescaling that will make the second commute! Because, if you are going to rescale $\xi$ by $x^a$ and $\psi$ by $x^b$ then you'd need $8a - 5b = 1$ and $5a = 2b$, which has no integer solution. But since the map is $\det \xi \otimes (\xi^{-1})^\vee$, when you rescale $\xi$ by $x$, the left map rescales by $x^3$ and there IS an integer solution to $8a - 5b = 1$ and $3a = 2b$ namely, $a = 2, b = 3$. This seems so special to me that this worked out now, like this linear system appeared that depended on the particular ranks and the fact that we're supposed to have $\det E^2 \cong \det F$, and I'm really impressed it had an integer solution. I guess that's part of the magic. I really hope this is correct now...}
\end{proof}

\subsection{Casnati--Ekedahl classes} \label{CEclass}
We now define some preferred generators for $R^*(\H_{k,g})$ using the Chern classes of vector bundles appearing in the Casnati--Ekedahl resolution. 
%The same constructions may be made with script font to define analogous rational classes on $\Hp_{k,g}$.
Let $\pi: \mathcal{P} \to \H_{k,g}$ denote the universal $\pp^1$-bundle and $\alpha: \C \to \P$ the universal degree $k$ cover. We define $z :=  - \frac{1}{2} c_1(\omega_\pi) = c_1(\O_{\P}(1))$ and \begin{equation} \label{cdef}
c_2 := c_2(\pi_* \O_{\P}(1)) \qquad \Rightarrow \qquad z^2 + \pi^* c_2 = 0,
\end{equation}
where the equality on the right follows from \eqref{pbt}.
%We now define some natural classes on $\H_{k,g}$ for $k = 3, 4, 5$ using the Chern classes of vector bundles appearing in the Casnati--Ekedahl resolution. Specifically, let $\pi: \mathcal{P} \to \H_{k,g}$ denote the universal $\pp^1$-bundle and $\alpha: \C \to \P$ the universal degree $k$ cover. We define $z := c_1(\O_{\P}(1))$ and \begin{equation} \label{cdef}
%c_2 := c_2(\pi_* \O_{\P}(1)) \qquad \Rightarrow \qquad z^2 = \pi^* c_2.
%\end{equation}
Define $\E^\vee := E_\alpha^{\vee}$ to be the cokernel of $\O_{\P} \to \alpha_* \O_{\C}$, which is a rank $k-1$ vector bundle on $\P$.  For $i = 1, \ldots, k-1$, we define classes $a_i \in A^i(\H_{k,g})$ and $a_i' \in A^{i-1}(\H_{k,g})$ by the formula
\begin{equation} \label{adef}
a_i := \pi_*(z \cdot c_i(\E)), \quad a_i' := \pi_*(c_i(\E)) 
\qquad \Rightarrow \qquad 
c_i(\E) = \pi^* a_i + \pi^* a_i' z .
\end{equation}
By Example \ref{p1ex}, $\E$ has relative degree $g + k-1$ on the fibers of $\P \to \H_{k,g}$, so $a_1' = g + k - 1$.
By the Casnati--Ekedahl structure theorem, the universal curve $\C$ embeds in $\pp \E^\vee$. We have the associated Casnati--Ekedahl resolution
\[
        0\rightarrow \gamma^*\mathcal{F}_{k-2}(-k)\rightarrow \gamma^*\mathcal{F}_{k-3}(-k+2)\rightarrow \cdots\rightarrow \gamma^*\mathcal{F}_{1}(-2)\rightarrow \mathcal{O}_{\p\E^{\vee}}\rightarrow \mathcal{O}_\C\rightarrow 0.
\]
For each bundle $\mathcal{F}_j$, we define 
\[
f_{j,i} := \pi_*(z \cdot c_i(\F_j)), \quad f_{j,i}' := \pi_*(c_i(\F_j)) 
\qquad \Rightarrow \qquad 
c_i(\F_j) = \pi^* f_{j,i} + \pi^* f_{j,i}' z .
\]
\begin{definition}\label{CEring}
We define $c_2, a_i, a_i', f_{j,i}, f_{j,i}'$ to be the \emph{Casnati--Ekedahl classes}, abbreviated CE classes.
\end{definition}

\begin{thm} \label{CEgen}
The CE classes are tautological and they generate the tautological ring $R^*(\H_{k,g})$.
\end{thm}
\begin{rem} \label{not-growing}
The ranks of the $\F_i$ depend only on $i$ and $k$, so this bounds the number of generators of $R^*(\H_{k,g})$ and their degrees in terms of $k$ (independent of $g$).
\end{rem}
\begin{proof}
First, we show that the Casnati--Ekedahl classes are tautological.
Let us call a class on $\P$ \emph{pre-tautological} if it is a polynomial in $z$ and classes of the form $\alpha_*(c_1(\omega_f)^j)$. By the push-pull formula, the $\pi$ pushforward of a pre-tautological class is tautological. Therefore, our goal is to show that the Chern classes of $\E$ and $\F_i$ are pre-tautological.

By Grothendieck--Riemann--Roch and the splitting principle, we have that the Chern classes of $\alpha_*(\omega_{\alpha}^{\otimes i}) = \alpha_*(\omega_f^{\otimes i}) \otimes (\omega_\pi^\vee)^{\otimes i}$ are pre-tautological. In particular, the Chern classes of $\E$ are pre-tautological by its defining exact sequence. By the construction of the Casnati--Ekedahl sequence, $\F_1$ is the kernel of a surjective map $\Sym^2 \E \twoheadrightarrow \alpha_*(\omega_{\alpha}^{\otimes 2})$, so the Chern classes of $\F_1$ are pre-tautological. Similarly, following the construction of $\F_i$ on \cite[p. 445-446]{CE} and using the splitting principle, we inductively see that the Chern classes of all $\F_i$ are pre-tautological.

Next, we must show that all tautological classes are polynomials in Casnati--Ekedahl classes. 
We have a diagram
\begin{equation} \label{unif}
\begin{tikzcd}
\C \arrow{dr}[swap]{\alpha} \arrow[ddr, bend right,swap, "f"] \arrow{r}[hookrightarrow]{\iota} & \pp \E^\vee \arrow{d}{\gamma} \\
& \P \arrow{d}{\pi} \\
& \H_{k,g}
\end{tikzcd}
\end{equation}
First, note that
\[f_*(c_1(\omega_f)^i \cdot \alpha^*(\omega_\pi)^j) = \pi_*(\alpha_*(c_1(\omega_\alpha) + \alpha^*c_1(\omega_\pi))^i \cdot c_1(\omega_\pi)^j), \]
so using push-pull, it will suffice to show that
$\pi_*(\alpha_*(c_1(\omega_\alpha)^i) \cdot z^j)$ is a polynomial in CE classes for all pairs $i, j$.
Now, let $\zeta := c_1(\O_{\pp \E^\vee}(1))$ and note that $\iota^*\zeta = c_1(\omega_\alpha)$. We have
\begin{align*}
\alpha_*(c_1(\omega_\alpha)^i) = \gamma_*\iota_*(\iota^* \zeta^i) = \gamma_*([\C] \cdot \zeta^i).
\end{align*}
Grothendieck--Riemann--Roch for $\iota: \C \hookrightarrow \pp \E^\vee$ tells us that $[\C] = \ch_{k-2}(\iota_* \O_\C)$.
By additivity of Chern characters in exact sequences,
the later is a polynomial in $\zeta$ and the Chern classes of $\F_i$.
Using the projective bundle theorem \eqref{pbt}, $\gamma_*([\mathcal{C}] \cdot \zeta^i)$ is therefore a polynomial in
the Chern classes of $\E$ and the $\F_i$. 
The $\pi$ push forward of such a polynomial times any power of $z$ is a polynomial in the CE classes (essentially from the definition of the CE classes).
\end{proof}

Using the idea in the proof above, we explain how to rewrite the $\kappa$-classes in terms of CE classes. 
\begin{example}[$\kappa$-classes] \label{kex}
Let us retain notation as in \eqref{unif}.
Writing $\zeta$ for the hyperplane class of $\p\E^{\vee}$ and $z$ for the hyperplane class on $\P$, we have
\[
c_1(\omega_f)= c_1(\omega_\alpha) + c_1(\omega_\pi) = \iota^*(\zeta-2z).
\]
By the push-pull formula, we have
\begin{equation} \label{kap}
\kappa_i=f_*(c_1(\omega_f)^{i+1})=\pi_*\gamma_*\iota_*(\iota^*(\zeta-2z)^{i+1})=\pi_*\gamma_*([\C]\cdot (\zeta-2z)^{i+1}).
\end{equation}
Meanwhile, the fundamental class of $\C \subset \pp \E^\vee$ is
\begin{align} \label{classC}
[\C] &= \sum_{i=1}^{k-3} (-1)^{i-1}\ch_{k-2}(\F_i(-i-1)) + (-1)^{k-2} \ch_{k-2}(\F_{k-2}(-k)) \\
&= \notag \begin{cases} -\ch_1(\det \E(-3)) = c_1(\det \E^\vee(3))  &\text{if $k = 3$} \\
-\ch_2(\F(-2)) + \ch_2(\det \E (-4))  = c_2(\F^\vee(2)) &\text{if $k = 4$} \\
-\ch_3(\F(-2)) + \ch_3((\F^\vee \otimes \det \E)(-3)) - \ch_3(\det \E(-5))
&\text{if $k = 5$.} \end{cases} 
\end{align}
Using \eqref{kap} and \eqref{classC}, it is straightforward to compute $\kappa_i$ in terms of the CE classes using a computer.
\end{example}

In degree $3$, the CE classes are $c_2, a_1, a_2, a_2'$. In degrees $k = 4, 5$, self-duality of the Casnati--Ekedahl resolution implies that all CE classes are expressible in terms of $c_2$ the $a_i, a_i'$ and  the $b_i := f_{1, i}$ and $b_i' := f_{1, i}'$. 
These classes are all pulled back from a moduli space $\B_{k,g}$ of (pairs of) vector bundles on $\pp^1$, which we shall construct in the next section.

\section{Pairs of vector bundles on $\pp^1$-bundles} \label{Vsec}
By the results of Casnati--Ekedahl and Casnati in the previous section, there is a correspondence between covers of $\p^1$ and certain linear algebraic data. In this section, following ideas of Bolognesi--Vistoli \cite{BV}, we construct moduli stacks parametrizing the associated linear algebraic data and describe the Chow rings of these stacks.
In \cite{BV}, Bolognesi--Vistoli gave a quotient stack presentation for the moduli stack parametrizing globally generated vector bundles on $\pp^1$-fibrations. As explained in Section \ref{hssec},
we will instead make use  of $\SL_2$ quotients, since they have the same rational Chow ring as the $\PGL_2$ quotient.
%All constructions that follow in later sections can be made over $\BPGL_2$ (represented in script font), but for convenience we will mostly work with the base change to $\BSL_2$ (in calligraphic font) so that the universal $\pp^1$-fibration is replaced with a universal $\pp^1$-bundle, which is the projectivization of vector bundle.
\begin{definition}\label{Stacks}
Let $r,d$ be nonnegative integers. 
\begin{enumerate}
    \item The objects of $\Ub_{r,d}^\dagger$ are pairs $(S, E)$ where $E$ is a locally free sheaf of rank $r$ on $\pp^1 \times S$ whose restriction to each of the fibers of $\pp^1 \times S \rightarrow S$ is globally generated of degree $d$. A morphism between objects $(S, E)$ and $(S', E')$ is a Cartesian diagram
    \[
    \begin{tikzcd}
\pp^1 \times S' \arrow[d] \arrow[r, "F"] & \pp^1 \times S \arrow[d] \\
S' \arrow[r]                & S          
\end{tikzcd}
    \]
    together with an isomorphism $\phi:F^*E\rightarrow E'$.
\item We define $\Ub_{r,d}$ to be the $\SL_2$ quotient of $\Ub_{r,d}^\dagger$. Explicitly, the objects of $\Ub_{r,d}$ are triples $(S, V, E)$ where $S$ is a $k$-scheme, $V$ is a rank $2$ vector bundle on $S$ with trivial determinant, and $E$ is a rank $r$ vector bundle on $\pp V$ whose restrictions to the fibers of $\pp V \rightarrow S$  are globally generated of degree $d$. A morphism between objects $(S,V,E)$ and $(S',V',E')$ is a Cartesian diagram
     \[
    \begin{tikzcd}
\pp V' \arrow[d] \arrow[r, "F"] & \pp V \arrow[d] \\
S' \arrow[r]                & S          
\end{tikzcd}
    \]
    together with an isomorphism $\phi:F^*E\rightarrow E'$.
\end{enumerate}
\end{definition}
\noindent

\par Bolognesi--Vistoli gave a presentation for $\Ub_{r,d}^\dagger$ as a quotient stack, which we briefly summarize here. Let $M_{r,d}$ be the affine space that represents the functor which sends a scheme $S$ to the set of matrices of size $(r+d)\times d$ with entries in $H^0(\p^1_S,\mathcal{O}_{\pp^1_S}(1))$. We can identify such a matrix with the associated map
\[
\mathcal{O}_{\p^1_S}(-1)^{d}\rightarrow \mathcal{O}_{\p^1_S}^{r+d}.
\]
Let $\Omega_{r,d} \subset M_{r,d}$ denote the open subscheme parametrizing injective maps with locally free cokernel. The group $\GL_d$ acts $M_{r,d}$ by multiplication on the left, $\GL_{r+d}$ by multiplication on the right. Bolognesi--Vistoli establish \cite[Theorem 4.4]{BV} that 
\[\Ub_{r,d}^\dagger \cong [\Omega_{r,d}/\GL_d \times \GL_{d+r}].\]
The group $\SL_2$ acts by change of coordinates on $H^0(\pp^1_S, \O_{\pp^1_S}(1))$. This commutes with the $\GL_d \times \GL_{d+r}$ action. Thus we obtain the following.

\begin{prop} \label{Vprime}
There is an isomorphism of fibered categories
\[
\Ub_{r,d}\cong [\Ub_{r,d}^\dagger/\SL_2] \cong [\Omega_{r,d}/\GL_{d}\times \GL_{r+d}\times \SL_2].
\]
\end{prop}

\begin{rem}
Bolognesi--Vistoli also describe the $\PGL_2$ quotient of $\Ub_{r,d}^\dagger$, which is slightly more subtle. This distinction is important in their work which concerns \emph{integral} coefficients.
\end{rem}

To parametrize the linear algebraic data associated to a low degree cover of $\pp^1$, we are interested in
products of the form $\Ub_{r,d} \times_{\BSL_2} \Ub_{s,e}$, which parametrize a pair of vector bundles on the same $\pp^1$-bundle. Let $G_{r,d,s,e} := \GL_d \times \GL_{r+d} \times \GL_e \times \GL_{s+e}$.
The group $G_{r,d,s,e} \times \SL_2$ acts on $M_{r,d}$ via the projection $G_{r,d,s,e} \times \SL_2 \to \GL_d \times \GL_{r+d} \times \SL_2$; and similarly on $M_{s,e}$ via the projection $G_{r,d,s,e} \times \SL_2 \to \GL_e \times \GL_{s+e} \times \SL_2$.
By Proposition \ref{Vprime}, it follows that
\begin{equation} \label{asq}
\Ub_{r,d} \times_{\BSL_2} \Ub_{s,e} = [\Omega_{r,d} \times \Omega_{s,e}/G_{r,d,s,e} \times \SL_2].
\end{equation}

Let $T_d$ and $T_{r+d}$ denote the universal vector bundles on $\BGL_d$ and $\BGL_{r+d}$; similarly, let $S_e$ and $S_{s+e}$ be the universal vector bundles on $\BGL_{e}$ and $\BGL_{s+e}$. 
The Chow ring of $\mathrm{B}(G_{r,d,s,e} \times \SL_2)$ is the free $\qq$-algebra on the Chern classes of $T_d, T_{r+d}, S_e, S_{s+e}$, together with the universal second Chern class $c_2$ on $\BSL_2$.
Let us denote these classes by
\begin{align*}
t_i &= c_i(T_d) \qquad \text{and} \qquad u_i = c_i(T_{r+d}) \\
v_i &= c_i(S_e) \qquad \text{and} \qquad  w_i = c_i(S_{s+e}).
\end{align*}
Since $\Omega_{r,d} \times \Omega_{s,e}$ is open inside the affine space $M_{r,d} \times M_{s,e}$, 
the excision and homotopy properties imply 
\begin{equation} \label{gens}
\text{$A^*(\Ub_{r,d} \times_{\BSL_2} \Ub_{s,e})$ is generated by the restrictions of the $t_i, u_i, v_i, w_i$.}
\end{equation}

We now identify the restrictions of the tautological bundles $T_d$ and $T_{d+r}$ in terms of the universal rank $r$, degree $d$ vector bundle on $\pp^1$.
Let $\pi: \mathcal{P} \rightarrow \Ub_{r,d}$ be the universal $\pp^1$-bundle. We write $z := c_1(\O_{\mathcal{P}}(1)) \in A^1(\mathcal{P})$.
We have $c_2 = c_2(\pi_* \O_{\mathcal{P}}(1)) \in A^2(\Ub_{r,d})$, the universal second Chern class (pulled back via the natural map $\Ub_{r,d} \rightarrow \BSL_2$).
Note that $c_1(\pi_* \O_{\mathcal{P}}(1)) = 0$, so by Equation \eqref{pbt},
\[A^*(\mathcal{P}) = A^*(\Ub_{r,d})[z]/(z^2 + \pi^*c_2).\]
Let $\E$ be the universal rank $r$, degree $d$ vector bundle on $\P$.
 The Chern classes of $\E$ may thus be written as
\[c_i(\E) = \pi^* a_i + (\pi^* a_i') z \qquad \text{where} \qquad a_i \in A^i(\Ub_{r,d}), \quad a_i' \in A^{i-1}(\Ub_{r,d}).\]
Note that $a_1' = d$. Let $\gamma: \Ub_{r,d} \times_{\BSL_2} \Ub_{s,e} \to B(G_{r,d,s,e} \times \SL_2)$ be the structure map. Then by \cite[Lemma 3.2]{Brd} (noting that $\det(\pi_* \O_{\P}(1))$ is trivial), we have 
\begin{equation} \label{idT}
\gamma^* T_d = \pi_* \E(-1) \qquad \text{and} \qquad \gamma^*T_{r+d} = \pi_* \E.
\end{equation}
Since $R^1\pi_*\E(-1)$ and $R^1\pi_*\E$ are zero, Grothendieck--Riemann--Roch says that
the Chern characters of $\pi_*\E(-1)$ and $\pi_*\E$ are push forwards by $\pi$ of polynomials in the $c_i(\E)$ and $z$. The push forward of such a polynomial is a polynomial in the $a_i, a_i'$ and $c_2$. In particular, the restrictions of $t_i$ and $u_i$ to $A^*(\Ub_{r,d})$ are polynomials in $a_1, \ldots, a_r, a_2', \ldots, a_r'$ and $c_2$. %(See Example \ref{exc1}.)

\subsection{The rational Chow ring}
Let us denote the universal rank $s$ vector bundle from the second factor of $\Ub_{r,d} \times_{\BSL_2} \Ub_{s,e}$ by $\F$ on $\P$ and its Chern classes by
\[c_i(\F) = \pi^* b_i + (\pi^* b_i') z \qquad \text{where} \qquad b_i \in A^i(\Ub_{s,e}), \quad b_i' \in A^{i-1}(\Ub_{s,e}).\]
It follows from Equation \eqref{gens} and the discussion following (applied to both factors of the product) that the $a_i, a_i', b_i, b_i'$ and $c_2$ are generators for $A^*(\Ub_{r,d} \times_{\BSL_2} \Ub_{s,d})$. We now show that there are no relations among these generators in low degrees.
This is a generalization of \cite[Theorem 4.1]{Brd}, which shows $A^*(\Ub_{r,d})$ is generated by the $a_i, a_i',$ and $c_2$ with no relations in degrees less than $d+1$.

\begin{thm} \label{Vchow}
The rational Chow ring of $\Ub_{r,d}\times_{\BSL_2} \Ub_{s,e}$ is generated as a $\qq$-algebra by
\[c_2, a_1, \ldots, a_r, a_2', \ldots, a_r', b_1, \ldots, b_s, b_2', \ldots, b_s',\]
and all relations have degree at least $\min(d, e)+1$. In the notation of Equation \eqref{barnot},
\begin{align*}
&\trun^{\min(d,e)+1} A^*(\Ub_{r,d} \times_{\BSL_2} \Ub_{s,e}) \\
& \qquad \qquad = \trun^{\min(d,e)+1}\qq[c_2, a_1, \ldots, a_r, a_2', \ldots, a_r',b_1, \ldots, b_s, b_2', \ldots, b_s'].
\end{align*}
%Moreover, if $d, e > 0$, then $A^1(\Ub_{r,d}\times_{\BSL_2} \Ub_{s,e}) \cong \zz a_1 \oplus \zz a_2' \oplus \zz b_1 \oplus \zz b_2'$ integrally.
\end{thm}
\begin{rem} (1) The codimension of the complement of $\Omega_{r,d} \subset M_{r,d}$ is $r$, so the Theorem does \emph{not} follow immediately from dimension counting and excision if $\min(d, e) > \min(r, s)$.

(2) If $s = 0$, the fact that there are no $b_i$ classes follows immediately from \cite[Theorem 4.1]{Brd}.
\end{rem}

\begin{proof}
Let 
\[M := [M_{r,d}/G_{r,d,s,e} \times \SL_2] \quad \text{and} \quad N := [M_{s,e}/G_{r,d,s,e} \times \SL_2].\]
Equation \eqref{asq} says that $\Ub_{r,d} \times_{\BSL_2} \Ub_{s,e}$ is an open inside the vector bundle $M \oplus N$ over $B := \mathrm{B}(G_{r,d,s,e} \times \SL_2)$.
The complement consists of two components, namely
\[X := [\Omega_{r,d}^c \times M_{s,e} /G_{r,d,s,e} \times \SL_2] \quad \text{and}
\quad 
Y :=  [M_{r,d} \times \Omega_{s,e}^c/G_{r,d,s,e} \times \SL_2].
\]
%Recall that $\Omega_{r,d}^c$ is the space of matrices of linear forms that drop rank along some point on $\pp^1$. 
One readily checks that $X \subset \Ub_{r,d} \times_{\BSL_2} \Ub_{s,e}$ is irreducible of codimension $r$ and $Y \subset \Ub_{r,d} \times_{\BSL_2} \Ub_{s,e}$ is irreducible of codimension $s$ (see \cite[Remark 4.3]{BV}).
Excision gives a right-exact sequence
\begin{equation} \label{removeXY}
A^{*-r}(X) \oplus A^{*-s}(Y) \rightarrow A^*(M \oplus N)  \rightarrow A^*(\Ub_{r,d} \times_{\BSL_2} \Ub_{s,e}) \rightarrow 0.
\end{equation}

From this it is clear that there are no relations among the restrictions to $\Ub_{r,d} \times_{\BSL_2} \Ub_{s,e}$ of the Chern classes of $T_d, T_{d+r}, S_e$ and $S_{s+e}$ in degrees less than $\min(r, s)$.
We now describe relations among the restrictions of these Chern classes in degrees $\min(r, s)$ up to $\min(d, e)$. (If $\min(d, e) < \min(r, s)$ we are already done.)
In particular, shall conclude that
\begin{align} \label{alt} &\trun^{\min(d,e)+1}A^*(\Ub_{r,d} \times_{\BSL_2} \Ub_{s,e}) \\
&\qquad = \trun^{\min(d,e)+1}\qq[c_2, t_1, \ldots, t_{r-1}, u_1, \ldots, u_r, v_1, \ldots v_{s-1}, 
w_1, \ldots, w_s]. \notag
\end{align}
Since the classes in the statement of the theorem are generators and have the same degrees as those above, the statement in the theorem must hold for dimension reasons.

It suffices to understand the image of $A^{*-r}(X) \to A^*(M \oplus N)$, the other factor being similar.
For this we resolve $X$ (resp. $Y$) as in the proof of \cite[Theorem 4.1]{Brd}, taking the $N$ factor (resp. $M$ factor) ``along for the ride".

Using the excision sequence \eqref{removeXY} and arguing exactly as in \cite[Theorem 4.1]{Brd}, we have
{\small \begin{equation} \label{quotient}
A^*(\Ub_{r,d} \times_{\BSL_2} \Ub_{s,e}) = \frac{\qq[c_2, t_1, \ldots, t_d, u_1, \ldots, u_{r+d}, v_1, \ldots, v_e, w_1, \ldots, w_{s+e}]}{\langle f_{i, j} : 0 \leq i \leq 1, 0 \leq j \leq d-1\rangle + \langle g_{i, j} : 0 \leq i \leq 1, 0 \leq j \leq e-1\rangle}.
\end{equation}}
where
\begin{align*}
f_{1, j-1} &= - t_{j + r} + u_{j + r} + \ldots & f_{0, j} &= -(r+d)t_{j+r} + (d - j)u_{j+r} + \ldots \\
g_{1, j-1} &=- v_{j + s} + w_{j + s} + \ldots & g_{0, j} &= -(s+e)v_{j+s} + (e - j)w_{j+s} + \ldots.
\end{align*}
Hence, in $A^*(\Ub_{r,d} \times_{\BSL_2}\Ub_{s,e})$,
the classes $t_{n}$ for  $r \leq n \leq d$ and $u_{m}$ for $r+1 \leq m \leq d$ are expressible as polynomials in $c_2, t_1, \ldots, t_{r-1}, u_1, \ldots, u_r$. Moreover, after eliminating these higher degree generators, the $f_{i,j}$ produce no additional relations in degrees less than or equal to $d$ among the restrictions to $\Ub_{r,d} \times_{\BSL_2} \Ub_{s,e}$ of $c_2, t_1, \ldots, t_{r-1}, u_1, \ldots, u_r, v_1, \ldots, v_{e}, w_1, \ldots, w_{s+e}$. With the analogous
calculation for the $g_{i,j}$, equation \eqref{quotient} then implies \eqref{alt}, and hence the statement of the theorem.
%To see the claim regarding the integral Picard group, note that if $d > 0$, then all relations in $J$ have degree greater than $1$. Hence, the integral Picard group is freely generated by $c_1(T_d), c_1(T_{r+d}), c_1(S_d),$ and $c_1(S_{r+d})$.
%Using Grothendieck--Riemann--Roch, we see
%\begin{align*}
%c_1(T_d) &= [\pi_*(\ch(\E)\ch(\O_{\P}(-1))\mathrm{Td}_{\pi})]_1 = [\pi_*(\ch(\E)(1 - z)(1 + z))]_1\\
%&= [\pi_*(\ch_2(\E))]_1 = da_1 - a_2' \\
%c_1(T_{r+d}) &= [\pi_*(\ch(\E)\mathrm{Td}_{\pi})]_1 = [\pi_*(\ch(\E)(1 + z))]_1 \\
%&= [\pi_*(\ch_2(\E))]_1 + [\pi_* \ch_1(\E)z]_1= da_1 - a_2'+ a_1.
%\end{align*}
%Similarly,
%\[c_1(S_d) = db_1 - b_2'\qquad \text{and} \qquad c_1(S_{r+d}) = (d+1)b_1 - b_2' \]
%so we see that $a_1, a_2', b_1$ and $b_2'$ are also integral generators for the Picard group.
%Because of the exact sequence
%\[0 \rightarrow \pi^*T_d \otimes \O_{\P}(-1) \rightarrow \pi^* T_{r+d} \rightarrow \E \rightarrow 0\]
%we have
%\begin{align*}
%1 + c_1(\E)  + c_2(\E) + \ldots &= \frac{c(\pi^*T_{r+d})}{c(\pi^*T_d \otimes \O_{\P}(-1)) } \\
%&= 1 + \pi^*c_1(T_{r+d}) - (\pi^*c_1(T_d) - dz) + \pi^*c_2(T_{r+d}) \\
%&\quad - 2c_1(T_{r+d})(\pi^*c_1(T_d) - dz) + (\pi^*c_1(T_d) - dz)^2
%\end{align*}
\end{proof}

\subsection{Splitting loci} \label{slsec}
Every vector bundle $E$ on $\pp^1$ splits as a direct sum of line bundles, $E \cong \O(e_1) \oplus \cdots \oplus \O(e_r)$ for integers $e_1 \leq \cdots \leq e_r$. We call the non-decreasing sequence of integers $\vec{e} = (e_1, \ldots, e_r)$ the \emph{splitting type} of $E$ and will often abbreviate the corresponding sum of line bundles by $\O(\vec{e}) := \O(e_1) \oplus \cdots \oplus \O(e_r)$.
Given a family of vector bundles $\E$ on a $\pp^1$-bundle $\pi: \P \to B$, the base $B$ is stratified by locally closed subvarieties
\[\{b \in B : \E_{\pi^{-1}(b)} \cong \O(\vec{e})\},\]
which we call the \emph{splitting locus for $\vec{e}$}. A subscheme structure on splitting loci is defined in \cite[Section 2]{L}, though it will not be necessary here.

The splitting type $\vec{e}$ of $E$ is equivalent to the data of the ranks of cohomology groups $h^0(\pp^1, E(m))$ for all $m \in \zz$. Conversely, the locus of points $b \in B$ where the fibers of $\E$ satisfy some cohomological condition is a union of splitting loci.
For example,
the locus in $B$ where $\E$ fails to be globally generated on fibers is the union of splitting loci for splitting types $\vec{e}$ with $e_1 \leq -1$. %This union of splitting loci is in turn equal to $\Supp R^1\pi_* \E(-1)$. 
Similarly, $\Supp R^1\pi_* \E(-2)$ is the union of all splitting loci with $e_1 \leq 0$.

Following the argument in \cite[Lemma 5.1]{BV}, the codimension in $\Ub_{r,d}$ of the splitting locus where the universal $\E$ over $\Ub_{r,d}$ has splitting type $\vec{e}$ on fibers of $\P \to \Ub_{r,d}$ is $h^1(\pp^1, \E nd(\O(\vec{e})))$.
If we have a $\pp^1$-bundle equipped with two vector bundles, we can consider the intersections of splitting loci for both bundles.
The simultaneous splitting locus in $\Ub_{r,d} \times_{\BSL_2} \Ub_{s,e}$ where $\E$ has splitting type $\vec{e}$ and $\F$ has splitting type $\vec{f}$ is equal to the product of the $\vec{e}$ splitting locus in $\Ub_{r,d}$ with the $\vec{f}$ splitting locus in $\Ub_{s,e}$, and therefore
has codimension 
\begin{equation} \label{simsp}
h^1(\pp^1, \E nd(\O(\vec{e}))) + h^1(\pp^1, \E nd(\O(\vec{f}))).
\end{equation}

\section{The good opens and codimension bounds}
For each $k = 3, 4, 5$ and genus $g$, we will define a stack $\B_{k,g}$ parametrizing the vector bundles associated to a degree $k$, genus $g$ cover of $\pp^1$. The stack $\B_{k,g}$ will come equipped with a universal $\pp^1$-bundle $\pi: \P \to \B_{k,g}$. Then, we will define a vector bundle $\aV_{k,g}$ on $\P$ whose sections on a fiber of $\P \to \B_{k,g}$ is the relevant space of sections in the linear algebraic data of covers appearing in Section \ref{CEsec}. 
%Many constructions in this section can be made over $\BPGL_2$ (using script letters) or over $\BSL_2$ (using caligraphic letters). We do not write out both, but it should be understood that a script letter means to make the same construction but replacing $\Ub_{r,d}$ with $\V_{r,d}$ in the initial step.

Before treating the case for each $k$ in depth, we briefly outline our construction of certain open substacks of the Hurwitz stack. For $k = 3$, we shall define $\B_{3,g}' \subseteq \B_{3,g}$ to be the open substack over which $\aV_{3,g}$ is globally generated on the fibers of $\pi: \P \to \B$.
%\hannah{I think we can get away with never saying $\Supp(R^1 \pi_*(\U_{3,g} \otimes \O_{\P}(-1)))$ and maybe that's better because people know what global generation means but might not be so quick to make the translation. The ``canonical" way to cut out the non (relatively) globally generated locus of a family is as the locus where $\pi^*\pi_* E \to E$ fails to be surjective, so I guess, i.e. the scheme theoretic image of the support of the cokernel of that map.}
For $k = 4, 5$, let us define the open substack
\begin{align}
\B_{k,g}' &:= \B_{k,g} \smallsetminus \Supp(R^1\pi_*\aV_{k,g}) \label{Bprime}.
\end{align}
By the theorem on cohomology and base change, the restriction of $\pi_* \aV_{k,g}$ to $\B_{k,g}'$ is locally free with fibers given by the relevant space of sections in the linear algebraic data of covers appearing in Section \ref{CEsec}. We denote the total space of this vector bundle on $\B_{k,g}'$ by 
\[\aW_{k,g}' := \pi_* \aV_{k,g}|_{\B_{k,g}'}.\]

As discussed in Section \ref{slsec}, the complement of $\B_{k,g}'$ is a union of splitting loci. The splitting loci in the complement are determined by the condition that $\U_{k,g}$ has a summand of degree $-2$ or less.
It will also be convenient to work with a slightly smaller open substack
\begin{equation}\label{Bcirc}
\B_{k,g}^\circ := \B_{k,g} \smallsetminus \Supp(R^1\pi_*(\aV_{k,g} \otimes \O_{\P}(-2))), 
\end{equation}
whose complement consists of splitting loci where $\U_{k,g}$ has a summand of degree $0$ or less. (Having a cut-off in terms of a degree $0$ summand will be cleaner than a cut-off in terms of a $-2$ summand for our next step of bounding the codimension of the complement. The open $\B_{k,g}^\circ$ also plays an important role in our sequel \cite{part2}, where 
the slightly stronger positivity condition on the fibers of $\U_{k,g}$ will be needed.)

Pulling back these open substacks along the natural map $\H_{k, g} \to \B_{k,g}$ defines open substacks of the Hurwitz space as in the diagram below
\begin{equation}
\begin{tikzcd} \label{kdiag}
\H_{k,g}^\circ \arrow{d} \arrow{r} &\H_{k,g}' \arrow{d} \arrow{r} & \H_{k,g} \arrow{d} \\
\B_{k,g}^\circ \arrow{r} &\B_{k,g}' \arrow{r} & \B_{k,g}.
\end{tikzcd}
\end{equation}
In all cases, we shall see that $\H_{k,g}'$ is an open substack inside the vector bundle $\aW_{k,g}'$ over $\B_{k,g}'$. In particular, we will find that the Chow ring of $\H_{k,g}'$ is generated by tautological classes. To turn this into meaningful results for the Chow ring of $\H_{k,g}$ we must describe the complement of $\H_{k,g}' \subseteq \H_{k,g}$. 
When $k = 3$, it turns out $\H_{k,g}' = \H_{k,g}$. For $k = 4$, the complement of $\H_{4,g}'$ contains covers that factor through an intermediate curve of low genus. Nevertheless, we show that the complement of $\H_{k,g}'$ intersects $\H_{4,g}^{\mathrm{nf}}$ in high codimension. 
Finally, for $k = 5$, we show that the complement of $\H_{5,g}'$ has high codimension.

Of course, the complement of $\H_{k,g}'$ is contained in the complement of $\H_{k,g}^\circ$, so it will suffice to bound the codimension of the complement of $\H_{k,g}^\circ$ (which in turn provides a lower bound on the codimension of the complement of $\H_{k,g}'$). For arbitrary $g$, the coefficient of $g$ in the bounds we obtain will be sharp (and the same for $\H_{k,g}^\circ$ and $\H_{k,g}'$).
In any particular case, however, one may find slight improvements on our bounds by enumerating the splitting loci in the complement $\H_{k,g}'$, and calculating their codimensions.

Along the way, we also bound the codimension of the complement of $\B_{k,g}^\circ$. We point out some immediate corollaries regarding the Chow rings of $\B_{k,g}^\circ$, which will be used in our subsequent work \cite{part2}.

\subsection{Degree $3$}
As a warm-up, we first explain our set-up in degree $3$, as it is simplest. The results in this subsection are not new (they have already been established by Bolognesi--Vistoli \cite{BV} and Patel--Vakil \cite{PV2}) but spelling them out in our language will be instructive; it will also be useful in our subsequent work \cite{part2}.

In Section \ref{Vsec}, we gave a construction for $\Ub_{r,d}$ as the moduli space of vector bundles on $\pp^1$-bundles. 
As discussed in Section \ref{tripsec}, the linear algebraic data of a degree $3$, genus $g$ cover involves a rank $2$, degree $g+2$ vector bundle $E$ on $\pp^1$ and section of $\det E^\vee \otimes \Sym^3 E$. We set 
\[\B_{3,g} := \Ub_{2,g+2} \qquad \text{and} \qquad 
\aV_{3,g} := \det \E^\vee \otimes \Sym^3 \E,\]
where $\E$ is the universal rank $2$ bundle on $\pi: \P \to \Ub_{2,g+2}$.
%Note that by Theorem \ref{Vchow}, $c_2, a_1, a_2, a_2'$ generate $A^*(\B_{3,g})$ and we have
%\begin{equation} \label{b3chow}
%A^*(\B_{3,g})|_{g+3} = \qq[c_2, a_1, a_2, a_2']|_{g+3}.
%\end{equation}
There is a natural map $\H_{3,g} \to \B_{3,g}$ that sends a family of triple covers $C \xrightarrow{\alpha} P \rightarrow S$ in $\H_{3,g}(S)$ to the associated rank $2$ vector bundle $E_\alpha$ on $P \to S$ in $\B_{3,g}(S)$. If $C \xrightarrow{\alpha} \pp^1$ is an integral triple cover and
$E_\alpha = \O(e_1) \oplus \O(e_2)$ is the associated rank $2$ vector bundle on $\pp^1$, then by \cite[Proposition 2.2]{BV}, we have $e_1, e_2 \geq \frac{g+2}{3}$. Equivalently, every summand of $\det E_\alpha^\vee \otimes \Sym^3 E_\alpha$ is non-negative. Hence, the map $\H_{3,g} \to \B_{3,g}$ factors through the substack $\B_{3,g}' \subseteq \B_{3,g}$ over which $\aV_{3,g}$ is globally generated on fibers of $\P \to \B_{3,g}$. In particular, $\H_{3,g}' = \H_{3,g}$.
We define $\aW_{3,g}' := \pi_*\aV_{3,g}|_{\B_{3,g}'}$, which is a vector bundle on $\B_{3,g}'$ by the theorem on cohomology and base change.

\begin{lem} \label{hopen}
There is an open inclusion $\H_{3,g} = \H_{3,g}' \to \aW_{3,g}'$. %Similarly, there is an open inclusion $\Hp_{3,g} \to \Wp_{3,g}$ where $\Wp_{3,g}$ is a vector bundle over $\V_{2,g+2}$, defined analogously using script letters.
In particular, $A^*(\H_{3,g})$ is generated by the CE classes $c_2, a_1, a_2, a_2'$, and therefore $A^*(\H_{3,g}) = R^*(\H_{3,g})$. 
\end{lem}
\begin{proof}
The first sentence was essentially observed in \cite[p. 12]{BV}. We include an explanation using our notation.
Given a scheme $S$, the objects of $\aW_{3,g}(S)$ are tuples $(P \to S, E, \eta)$ where $(P \to S, E)$ is an object of $\B_{3,g}'(S)$ and $\eta \in H^0(P, \Sym^3 E \otimes \det E^\vee)$. We define an open substack $\mathcal{Y}_{3,g}' \subset \aW_{3,g}'$ by the condition that $V(\Phi(\eta)) \subset \pp E^\vee \to S$ is a family of smooth curves, where $\Phi$ is as in \eqref{3com}. Considering the Hilbert polynomial of $V(\Phi(\eta))$, one sees that the fibers have arithmetic genus $g$.
Theorem \ref{CE3} now shows that there is an equivalence $\H_{3,g}' \cong \mathcal{Y}_{3,g}'$.

By excision, the Chow ring of $\H_{3,g} = \H_{3,g}'$ is generated by restrictions of classes on $\aW_{3,g}'$. Since $\aW_{3,g}'$ is a vector bundle over $\B_{3,g}'$, their Chow rings are isomorphic, so the statement about generators follows from Theorem \ref{Vchow}.
\end{proof}

\subsection{Degree $4$} \label{co4}

By Casnati--Ekedahl's characterization of quadruple covers (Theorem \ref{CE4}), the linear algebraic data of a quadruple cover of $\pp^1$ is equivalent to the data of: a rank $3$ vector bundle $E$; a rank $2$ vector bundle $F$; an isomorphism $\det F \cong \det E$; and a global section of $F^\vee \otimes \Sym^2 E$ on $\p^1$ having the right codimension. By Example \ref{p1ex}, $\deg(E) = \deg(F) = g + 3$.
The stacks $\Ub_{2,g+3}$ and $\Ub_{3,g+3}$ both admit natural morphisms to $\BSL_2$, and the fiber product $\Ub_{3,g+3}\times_{\BSL_2} \Ub_{2,g+3}$ is the stack whose objects are quadruples $(S,V,E,F)$ where $S$ is a $k$-scheme, $V$ is a rank $2$-vector bundle with trivial determinant, $E$ is a rank $3$ vector bundle on $\p V$ whose restriction to the fibers of $\p V\rightarrow S$ is globally generated of degree $g+3$, and $F$ is a rank $2$ vector bundle on $\p V$ whose restriction to the fibers of $\p V\rightarrow S$ is globally generated of degree $g+3$.

The additional data of an isomorphism $\det F \cong \det E$ is captured by a $\mathbb{G}_m$ torsor over $\Ub_{3,g+3} \times_{\BSL_2} \Ub_{2,g+3}$ defined as follows.
Let $\E$ be the universal rank $3$ bundle and $\F$ be the universal rank $2$ bundle on the universal $\pp^1$-bundle $\pi: \mathcal{P} \rightarrow \Ub_{3,g+3}\times_{\BSL_2} \Ub_{2,g+3}$.
Since $\det \E^\vee \otimes \det \F$ has degree $0$ on each fiber of $\pi: \P \rightarrow \Ub_{3,g+3}\times_{\BSL_2} \Ub_{2,g+3}$, the theorem on cohomology and base change shows that
$\mathcal{L}:=\pi_*(\det \E^\vee \otimes \det \F)$
is a line bundle with $\pi^*\mathcal{L} \cong \det \E^\vee \otimes \det \F$.
\begin{definition}
With notation as above, define the stack $\B_{4,g}$ to be the $\mathbb{G}_m$-torsor over $\Ub_{3,g+3}\times_{\BSL_2} \Ub_{2,g+3}$  given by the complement of the zero section of the line bundle $\mathcal{L}$.
\end{definition}
The objects of $\B_{4,g}$ are tuples $(S, V, E, F, \phi)$ where $(S, V, E, F)$ is an object of $\Ub_{3,g+3} \times_{\BSL_2} \Ub_{2,g+3}$ and $\phi$ is an isomorphism $\det F \cong \det E$.
Recalling the notation of Section \ref{quadsec}, given an object $C \xrightarrow{\alpha} P \to S$ of $\H_{4,g}(S)$, the restriction of $E_\alpha$ and $F_\alpha$ to fibers of $P \to S$ are both known to be globally generated (see Proposition \ref{known}).
Hence,
there is a natural map $\H_{4,g} \to \B_{4,g}$ that sends the family $C \xrightarrow{\alpha} P \xrightarrow{\pi} S$ to the tuple $(S, \pi_*\O_{P}(1)^{\vee}, E_\alpha, F_\alpha, \phi_\alpha)$.

By slight abuse of notation, let us denote the pullback to $\B_{4,g}$ of the universal $\pp^1$-bundle by $\pi: \P \rightarrow \B_{4,g}$, and the universal rank $3$ and $2$ vector bundles on it by $\E$ and $\F$.
Let $z = \O_{\P}(1)$ and write
\[c_i(\E) = \pi^*a_i + (\pi^*a_i') z \qquad \text{and} \qquad c_i(\F) = \pi^*b_i + (\pi^*b_i') z. \]
for $a_i, b_i \in A^i(\B_{4,g})$ and $a_i', b_i' \in A^{i-1}(\B_{4,g})$.
Note that $a_1' = b_1' = g+3$.
Moreover, by definition of $\B_{4,g}$, we have $c_1(\det \E^\vee \otimes \det \F) = 0$, so $a_1 = b_1$.
Further, by Lemma \ref{Gmbundles}, we have
\begin{align} \label{b4eq}
A^*(\B_{4,g}) = A^*(\Ub_{3,g+3} \times_{\BSL_2} \Ub_{2,g+3})/\langle c_1(\mathcal{L})\rangle = A^*(\Ub_{3,g+3} \times_{\BSL_2} \Ub_{2,g+3})/\langle a_1 - b_1\rangle.
\end{align}
Thus, Theorem \ref{Vchow}
shows that $c_2, a_1, a_2, a_3, a_2', a_3', b_2', b_2$ generate $A^*(\B_{4,g})$ and
\begin{equation} \label{b4chow}
\trun^{g+4} A^*(\B_{4,g}) = \trun^{g+4} \qq[c_2, a_1, a_2, a_3, a_2', a_3', b_2', b_2].
\end{equation}

Next, we define $\aV_{4,g} := \F^\vee \otimes \Sym^2 \E$ on $\P$. We then define $\B_{4,g}'$ and $\B_{4,g}^\circ$ by \eqref{Bprime} and \eqref{Bcirc} respectively.
Correspondingly, the open substacks $\H_{4,g}^\circ \subseteq \H_{4,g}' \subseteq \H_{4,g}$ are described by
\begin{align*} 
\{S \rightarrow \H_{4,g}^\circ\} &= \{S \rightarrow \H_{4,g} : R^1(\pi_S)_* (\F_S^\vee \otimes \Sym^2 \E_S \otimes \O_{\P_S}(-2)) = 0\}  \\
\{S \rightarrow \H_{4,g}'\} &= \{S \rightarrow \H_{4,g} : R^1(\pi_S)_* (\F_S^\vee \otimes \Sym^2 \E_S) = 0\}.
\end{align*}
The key property of $\H_{4,g}'$ is that the map $\H_{4,g}' \to \B_{4,g}'$ factors through an open inclusion in the total space of a vector bundle $\aW_{4,g}' := \pi_*\aV_{4,g}|_{\B_{4,g}'}$.

\begin{lem} \label{Hprime4}
There is an open inclusion $\H_{4,g}' \to \aW_{4,g}'$. %, and similarly $\Hp'_{4,g} \to \Wp_{4,g}$.
In particular, $A^*(\H_{4,g}') = R^*(\H_{4,g}')$ is generated by the CE classes $c_2, a_1, a_2, a_3, a_2', a_3', b_2', b_2$. 
\end{lem}
\begin{proof}
The objects of $\aW_{4,g}'$ are tuples $(S,V, E, F, \phi, \eta)$ where  $(S,V, E, F, \phi) \in \B_{4,g}'$ and $\eta \in H^0(\pp V, F^\vee \otimes \Sym^2 E)$. Letting $\Phi$ be as in \eqref{4comp}, we define $\mathcal{Y}_{4,g} \subset \aW_{4,g}'$ to be the open substack defined by the condition that $V(\Phi(\eta)) \subset \pp E^\vee \to S$ is a family of smooth curves. Considering the Hilbert polynomial of $V(\Phi(\eta))$, using \eqref{deg4res}, we see that the fibers have arithmetic genus $g$.
Using Theorem \ref{CE4}, we see that $\H_{4,g}'$ is equivalent to $\mathcal{Y}_{4,g}'$

%If $V$ is a rank $2$ vector bundle on $S$ with trivial determinant, then $\O_{\pp V}(-2) \cong \omega_{\pp V/S}$.
%Therefore, if we wish to work with $\pp^1$-fibrations, we replace all caligraphic letters with the same script letters and $\O_{\P}(-2)$ with $\omega_{\pi}$ in defining $\Hp_{4,g}^\circ$.

By excision, the Chow ring of $\H_{4,g}'$ is generated by restriction of classes from $\aW_{4,g}'$. Since $\aW_{4,g}'$ is a vector bundle over $\B_{4,g}'$, their Chow rings are isomorphic, so the statement about generators follows from \eqref{b4eq}.
\end{proof}

\begin{lem} \label{b4}
The codimension of $\Supp (R^1 \pi_*(\aV_{4,g} \otimes \O_{\P}(-2)))$ is at least $\frac{g+3}{4} - 4$. That is, the codimension of the complement of $\B_{4,g}^\circ \subseteq \B_{4,g}$ has codimension at least $\frac{g+3}{4} - 4$.
\end{lem}
\begin{proof}
By equation \eqref{simsp},
the  codimension of the support of $R^1 \pi_*(\F^\vee \otimes \Sym^2 \E \otimes \O_{\P}(-2))$ is the minimum value of 
$h^1(\pp^1, \E nd(\O(\vec{e}))) + h^1(\pp^1, \E nd(\O(\vec{f}))$
as we range over splitting types $\vec{e} = (e_1, e_2, e_3)$ with $e_1 \leq e_2 \leq e_3$ and $\vec{f} = (f_1, f_2)$ with $f_1 \leq f_2$ and 
\[h^1(\pp^1, \O(\vec{f})^\vee \otimes (\Sym^2 \O(\vec{e})) \otimes \O_{\pp^1}(-2)) > 0 \qquad \Leftrightarrow \qquad 2e_1 \leq f_2.\]
We have
\[h^1(\pp^1, \E nd(\O(\vec{e}))) + h^1(\pp^1, \E nd(\O(\vec{f}))) \geq 2e_3 - 2e_1 - 3 + f_2 - f_1 - 1.\]
To find the minimum, we consider the function of $5$ real variables
\[f(x_1, x_2, x_3, y_1, y_2) := 2x_3 - 2x_1 + y_2 - y_1 \]
on the compact region $D$ defined by 
\begin{gather*}
0 \leq x_1 \leq x_2 \leq x_3, \quad x_1 + x_2 + x_3 = 1, \quad y_1 \leq y_2, \quad y_1 + y_2 = 1, \quad 2x_1 \leq y_2.
\end{gather*}
Since $f$ is piecewise linear, its extreme values are attained where multiple boundary conditions intersect at a point. Code provided at \cite{github} performs the linear algebra to locate such points and evaluates $f$ at the them to determine its minimum. The minimum is $\frac{1}{4}$, attained at $(\frac{1}{4}, \frac{3}{8}, \frac{3}{8}, \frac{1}{2}, \frac{1}{2})$.
Thus,
\[\dim \Supp R^1 \pi_*(\aV_{4,g} \otimes \O_{\P}(-2)) \geq (g+3) \cdot \min_D (f) - 4 = \frac{g+3}{4} - 4. \qedhere  \]
\end{proof}

For later use, let us note an immediate consequence of the previous lemma: Using excision and \eqref{b4}, we see
\begin{equation} \label{b4forlater}
\trun^{(g+3)/4 - 4} A^*(\B_{4,g}^\circ) = \trun^{(g+3)/4 - 4} \qq[c_2, a_1, a_2, a_3, a_2', a_3', b_2', b_2].
\end{equation}

Just because the complement of $\B_{4,g}^\circ$ has high codimension inside $\B_{4,g}$ does \textit{not} mean that the complement of $\H_{4,g}^\circ$ will have high codimension in $\H_{4,g}$. 
The condition for $\alpha: C \to \pp^1$ to be in $\H_{4,g}^\circ$ is that $h^1(\pp^1, F_\alpha^\vee \otimes \Sym^2 E_\alpha) =0$. We shall refer to this as ``our cohomological condition." Our cohomological condition fails for factoring covers, as we explain now.
Suppose $\alpha: C \rightarrow \pp^1$ factors as $C \xrightarrow{\beta} C' \xrightarrow{h} \pp^1$ where $C'$ has genus $g'$. We claim $E_\alpha = \O(g'+ 1) \oplus E'$ for some rank $2$ bundle $E'$. Indeed, because $\beta$ is a double cover, we have 
\[
\beta_*\O_{C}\cong \O_{C'}\oplus L,
\]
where $L$ is a line bundle on $C'$. Pushing forward again by $h$,
\[
\alpha_*\O_C\cong \O_{\p^1}\oplus \O_{\p^1}(-g'-1) \oplus h_* L.
\]
This establishes that $E_\alpha$ has an $\O(g'+1)$ summand. In particular, since some summand of $F$ has degree at least $\frac{g+3}{2}$, 
\[h^1(\pp^1, F^\vee \otimes \Sym^2 E) \geq \frac{g+3}{2} - 2(g'+1) - 1.\]
Thus, covers that factor with $g'$ small are never in $\H^\circ_{4,g}$.
More precisely, if a factoring cover does satisfy our cohomological condition, then the genus of the intermediate curve must satisfy $2(g' + 1) \geq \frac{g+3}{2}$.

\begin{lem} \label{nfshave}
The locus of degree $4$ covers $C \to \pp^1$ that factor $C \rightarrow C' \rightarrow \pp^1$ where $C'$ has genus $g'$ has codimension $2(g'+1)$ in $\H_{4,g}.$ Hence, the complement of $\H^{\circ}_{4,g} \cap \H^{\nf}_{4,g} \subset \H^{\circ}_{4,g}$ has codimension at least $\frac{g+3}{2}$.
\end{lem}
\begin{proof}
The dimension of the Hurwitz stack is the degree of the branch locus minus $3 = \dim \Aut(\pp^1)$, giving $\dim \H_{4,g} = 2g + 3$. Meanwhile, by Riemann--Hurwitz, the dimension of the space of genus $g$ double covers of a fixed curve $C'$ of genus $g'$ is $2g - 2 - 2(2g' - 2)$. The dimension of the stack of genus $g'$ double covers of $\pp^1$ modulo $\Aut(\pp^1)$ is $2g' - 1$. Therefore, the dimension of the space of degree $4$ covers that factor through a curve of genus $g'$ is
\[2g - 2 - 2(2g' - 2) + 2g' - 1 = 2g + 1 - 2g' = \dim \H_{4,g} - 2(g' + 1).\qedhere\]
\end{proof}

Covers that factor through a curve of low $g'$ are therefore loci of fixed codimension that fail our cohomological condition. For this reason, in degree $4$, our techniques will only prove that certain Chow groups of the locus of \textit{non-factoring} covers are generated by tautological classes.
Below, we collect some results about the splitting types of the vector bundles associated to a degree $4$ cover. These facts were known to Schreyer \cite{S} (though Schreyer's notation differs from ours). We include proofs here as they demonstrate the geometric meaning of splitting types.

\begin{prop} \label{known}
Suppose $\alpha: C \rightarrow \pp^1$ is a degree $4$ cover and $E_\alpha = \O(e_1) \oplus \O(e_2) \oplus \O(e_3)$ with $e_1 \leq e_2 \leq e_3$, and $F = \mathcal{O}(f_1)\oplus\mathcal{O}(f_2)$ with $f_1\leq f_2$. The following are true:
\begin{enumerate}
    \item $e_1+e_2+e_3=f_1+f_2 =g+3$ and with $e_1 \geq 1$ if $C$ irreducible. %(Lazarsfeld, Prop. 1.2)
    \item If $C$ is irreducible, $2e_1\geq f_1$, %(Anands, ) 
    and $2e_2 \geq f_2$. %(Schreyer, p. 128). 
     Hence $F$ is globally generated. \label{DP}
%    \item \label{ob} If $\alpha$ does not factor then $m_3-m_2\leq m_1$ and $m_2-m_1\leq m_1$. [It seems this is not necessary]
    \item If $\alpha$ does not factor then $e_1 + e_3 - f_2 \geq 0$.
    \label{6}
%    \item If $m_1 + m_3 - b < 0$ then $a = 2m_1$. [I think useful for low genus things, but not necessary for coh condition lemma]
\end{enumerate}
\end{prop}
%Refs:
%R. Lazarsfeld, A barth-Type theorem for branched covers of Projective space, Math. An.. 249, 154--162 (1980)
%F. O. Schreyer, Syzygies of Canonical Curves and Special Linear Series, Math. Ann. 275, 105--137 (1986)
%\begin{remark}
%We are following notation of Anand and Anand. The translation to Schreyer is $m_3 = e_1 + 2, m_2 = e_2 + 2, m_1 = e_3 + 2, b = b_1 + 4, a = b_2 + 4$. One confusing thing is Schreyer writes $b_2 \geq -1$. This would say that $a \geq 3$. But I think this is because Schreyer is thinking about things where the gonality is exactly $4$. This stems from something earlier where Schreyer has $e_3 \geq 0$, which would say $m_1 \geq 2$, which must follow from a gonality exactly $4$ assumption.
%\end{remark}
\begin{proof}
(1) follows from Example \ref{p1ex} and fact that $\det E_\alpha\cong \det F_\alpha$. If $C$ is irreducible, we have $h^0(\pp^1, E_\alpha^\vee) = h^0(\pp^1, \alpha_* \O_C) - 1 = 0$, so $e_1 \geq 1$.

The remaining conditions can be seen from the description of $C$ as the intersection of two relative quadrics on $\pp E_\alpha^\vee$. Let us choose a splitting $E = \O(e_1) \oplus \O(e_2) \oplus \O(e_3)$ and corresponding coordinates $X, Y, Z$ on $\pp E^\vee$. The two quadrics that define $C$ are of the form
\begin{align}
p &= p_{1,1}X^2 + p_{1,2} XY + p_{2,2} Y^2 + p_{1,3} XZ + p_{2,3} YZ + p_{3,3} Z^2 \label{peq} \\
q &= q_{1,1}X^2 + q_{1,2} XY + q_{2,2} Y^2 + q_{1,3} XZ + q_{2,3} YZ + q_{3,3}, Z^2 \label{qeq}
\end{align}
where $p_{i,j}$ is a polynomial on $\pp^1$ of degree $e_i + e_j - f_1$ and $q_{i,j}$ is a polynomial on $\pp^1$ of degree $e_i + e_j - f_2$. If this degree is negative, then we mean this coefficient is zero. 

(2) If $2e_1< f_1$, then $p_{1,1} = q_{1,1} = 0$ and $C = V(p, q)$ would contain the curve $Y = Z = 0$, forcing $C$ to be reducible. If $2e_2 < f_2$, then $q_{1,1} = q_{1,2} = q_{2,2} = 0$ so $Z$ divides $q$. If $C$ were irreducible, it would be contained in one of the linear components of $V(q)$ but this is impossible. The global generation of $F$ follows because the inequalities imply $f_1 = g + 3 - f_2 \geq e_1 \geq 1$.

%(3) is a theorem of Ohbuchi \hannah{Ref needed}
%the last inequliaty follows from the string of inequalities 
%\[2m_2 \geq m_1 + m_2 \geq m_3 = g+3 - (m+1 + m_2) \geq g + 3 - 2m_2.\]

(3) If $e_1 + e_3 - f_2 \leq -1$, then we show $\alpha$ factors. This inequality implies
\[2e_1 - f_2 \leq e_1 + e_2 - f_2 \leq e_1 + e_3 - f_2 \leq -1,\]
so the coefficients $p_{1,1}, p_{1,2}$, and $p_{1,3}$ vanish. Therefore, $p$ is a combination of $Y^2, YZ,$ and $Z^2$. Hence, $V(p)$ is reducible in every fiber and contains the point $[1, 0, 0]$ in each fiber. 
\begin{center}
\includegraphics[width=3in]{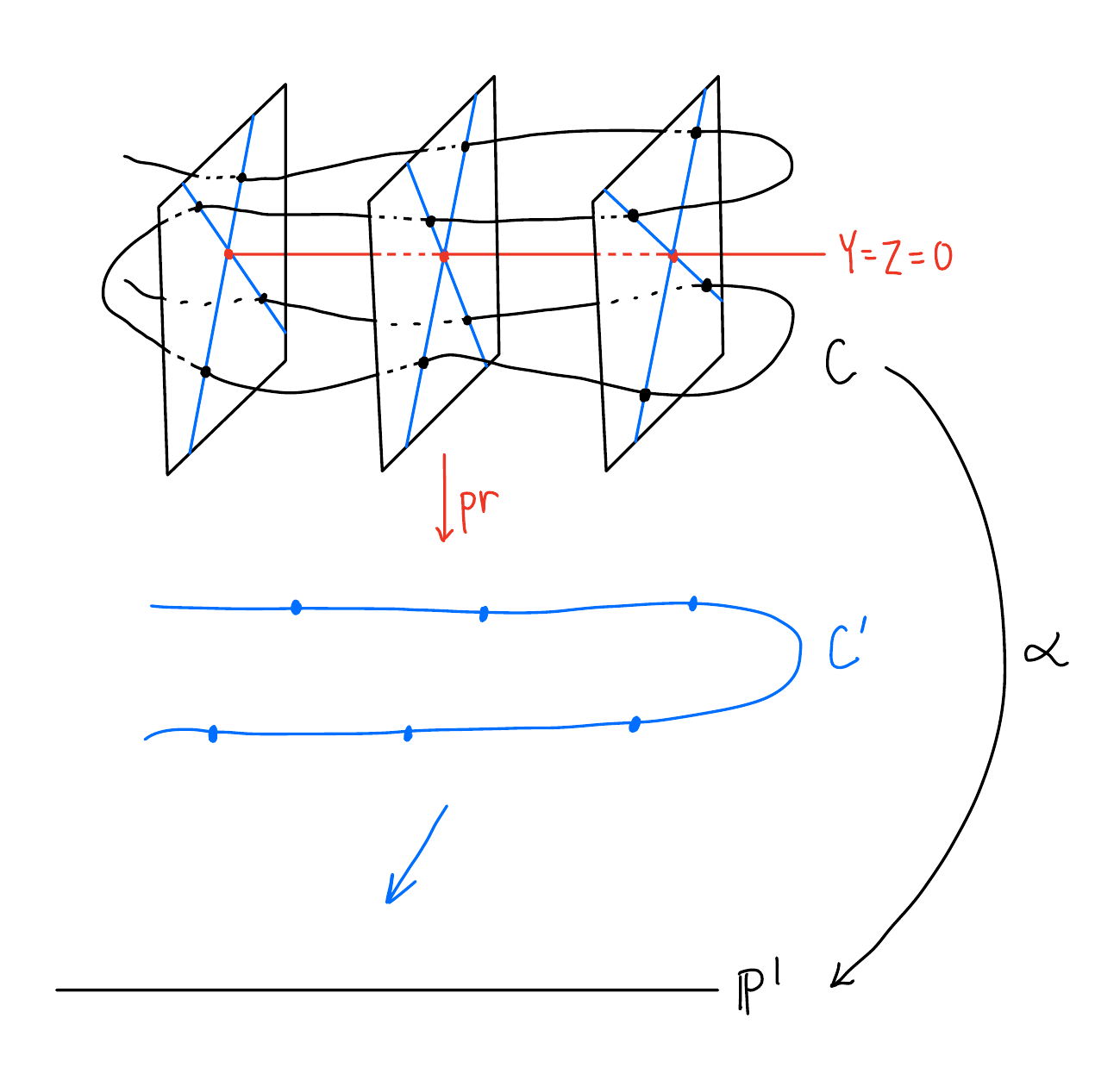}
\end{center}
In other words, each fiber of $C \to \pp^1$ consists of two pairs of points colinear with $[1, 0, 0]$.
Projection away from the line $Y = Z = 0$ defines a double cover $C \to C'$ that factors $\alpha$. 
\end{proof}

The simultaneous splitting loci of the universal $\E$ and $\F$ over $\H_{4,g}$ give rise to a stratification of $\H_{4,g}$.
In \cite[Remark 4.2]{DP}, Deopurkar--Patel show that the codimension of the splitting locus where $\E$ has splitting type $\vec{e}$ and $\F$ has splitting type $\vec{f}$ is
\begin{equation} \label{4codim}
h^1(\pp^1, \E nd(\O(\vec{e})))+h^1(\pp^1, \E nd(\O(\vec{f})))-h^1(\pp^1, \O(\vec{f})^\vee \otimes \Sym^2  \O(\vec{e})).
\end{equation}
Note that this differs from \eqref{simsp} by $h^1(\pp^1, \O(\vec{f})^\vee \otimes \Sym^2  \O(\vec{e}))$!

\begin{example}[$g = 6$]
We have $\dim \H_{4,6} = \dim \M_6 = 15$.
Using Proposition \ref{known} (2), we see that the non-empty strata are
\begin{enumerate}
    \item $\vec{e} = (3, 3, 3), \vec{f} = (4,5)$, (codimension $0$): The generic stratum.
    \item $\vec{e} = (2, 3, 4), \vec{f} = (4,5)$, (codimension $1$): By Casnati-Del Centina \cite{CDC}, the bielliptic locus is contained in this stratum as the locus where $p_{1,2} = 0$ and $p_{1,3} = 0$. Note that $\deg(p_{1,2}) = 0$ and $\deg(p_{1,3}) = 1$, so this represents $3$ conditions, making the bielliptic locus codimension $4$ inside $\H_{4,6}$.
    \item $\vec{e} = (3, 3, 3), \vec{f} = (3,6)$, (codimension $2$): This stratum consists of trigonal curves. We have $\pp E^\vee \cong \pp^1 \times \pp^2$. Since $\deg(q_{i,j}) =0$ and $\deg(p_{i,j}) = 3$ for all $i, j$, the projection onto the $\pp^2$ factor realizes $C$ as a degree $3$ cover of a conic in $\pp^2$.
    \item $\vec{e} = (2, 3, 4), \vec{f} = (3, 6)$, (codimension 2): Curves with a $g^2_5$. We have $p_{1,1} = 0$ and $\deg(q_{1,1}) = 1$, so the curve meets the line $Y = Z = 0$ in $\pp E^\vee$ in one point, say $\nu \in C$. The canonical line bundle on $C$ is the restriction of $\O_{\pp E^\vee}(1) \otimes \omega_{\pp^1}$, which contracts the line $Y = Z = 0$ in the map $\pp E^\vee \to \pp^5$. Thus, $\nu$ is contained in each of the planes spanned by the image of a fiber of $\alpha$ under the canoncial. Hence, the $g^1_4$ plus $\nu$ is a $g^2_5$. The locus of genus $6$ curves possessing a $g^2_5$ is codimension $3$ in $\M_6$, but this stratum has codimension $2$ in $\H_{4,6}$ because projection from any point on a plane quintic gives a $g^1_4$.
    \item $\vec{e} = (1, 4, 4), \vec{f} = (2, 7)$, (codimension $2$): Hyperelliptic curves
\end{enumerate}
The open $\H_{4,6}'$ is the union of strata (1), (2), and (3), while $\H_{4,6}^{\circ}$ contains only the generic stratum (1).
The image in $\Mg$ of $\H_{4,6}'$ under the forgetful map is what Penev--Vakil \cite{PV} call the ``Mukai general locus" of genus $6$ curves.
\end{example}

Using the numerical results of Lemma \ref{known}, we show that the codimension of \textit{non-factoring} covers that fail our cohomological condition grows as a positive fraction of the genus. 
\begin{lem} \label{coh-4}
The locus of non-factoring degree $4$ covers $\alpha: C \rightarrow \pp^1$ such that
\[h^1(\pp^1, F_\alpha^\vee \otimes \Sym^2 E_\alpha \otimes \O(-2)) > 0\]
has codimension at least $\frac{g+3}{4} - 4$. That is, the codimension the complement of $\H^\circ_{4,g} \cap \H^{\nf}_{4,g} \subset \H^{\nf}_{4,g}$ is at least $\frac{g+3}{4} - 4$.
\end{lem}
\begin{proof}
By equation \eqref{4codim}, the codimension of the locus of covers $\alpha$ with $E_\alpha = \O(\vec{e})$ and $F_\alpha = \O(\vec{f})$ 
is
\begin{align*}
u(\vec{e}, \vec{f}) &:= h^1(\pp^1, \E nd(\O(\vec{e}))) + h^1(\pp^1, \E nd(\O(\vec{f}))) - h^1(\pp^1, \O(\vec{f})^\vee \otimes \Sym^2 \O(\vec{e})) \\
&\geq 2e_3 - 2e_1 + f_2 - f_1 - 4 - h^1(\pp^1, \O(\vec{f})^\vee \otimes \Sym^2 \O(\vec{e})).
\end{align*}
Assuming $\alpha$ does not factor, Proposition \ref{known} (2) and (3) show that the only summands of $\O(\vec{f})^\vee \otimes \Sym^2\O(\vec{e})$ that can contribute to $h^1(\pp^1, \O(\vec{f})^\vee \otimes \Sym^2 \O(\vec{e}))$ are $\mathcal{O}(2e_1-f_2)$ and $\mathcal{O}(e_1+e_2-f_2)$.
Thus, our task is to bound the function
\begin{align*}
    &2e_3 - 2e_1 + f_2 - f_1 - 4 - \max\{0, f_2 - 2e_1 - 1\} - \max\{0, f_2 - e_1 - e_2 - 1\}
    \end{align*}
from below on the region where the conditions of Proposition \ref{known} hold and $2e_1 -f_2 \leq 0$, which is equivalent to $h^1(\pp^1, \O(\vec{f})^\vee \otimes \Sym^2\O(\vec{e}) \otimes \O(-2)) > 0$.

Let us introduce a function of $5$ real variables
\begin{align*}
f(x_1, x_2, x_3, y_1, y_2) &:= 2x_3 - 2x_1 + y_2 - y_1 - \max\{0, y_2 - 2x_1\} - \max\{0, y_2 - x_1 - x_2\}
\end{align*}
so that
\[u(\vec{e}, \vec{f}) \geq (g+3) f\left(\frac{e_1}{g+3}, \frac{e_2}{g+3}, \frac{e_3}{g+3},\frac{f_1}{g+3}, \frac{f_2}{g+3}\right) - 4.\]
We wish to minimize $f$ on the compact region defined by %\hannah{also added below $y_2 \leq 2x_2$. This only makes the region smaller...just adding it so someone doesn't wonder why we  left it out}
\begin{gather*}
 x_1 + x_2 + x_3 = 1, \quad y_1 + y_2 = 1, \quad
 0 \leq x_1 \leq x_2 \leq x_3, \quad
0 \leq y_1 \leq 2x_1 \leq y_2 \leq 2x_2, x_1 + x_3.
\end{gather*}
These correspond to the conditions from Proposition \ref{known}, together with the condition that $2e_1 \leq f_2$, which must be satisfied if the cohomological condition is failed. 
Since $f$ is piecewise linear, its extreme values are attained where multiple boundary conditions (including those where the function changes) intersect at a point.
A program provied in \cite{github} performs the linear algebra to locate such points and evaluates $f$ at them to determine its minimum. The minimum is $\frac{1}{4}$, attained at $(x_1, x_2, x_3, y_1, y_2) = (\frac{1}{4}, \frac{3}{8}, \frac{3}{8}, \frac{1}{2}, \frac{1}{2})$. It follows that, if $h^1(\pp^1, \O(\vec{f})^\vee \otimes \Sym^2\O(\vec{e}) \otimes \O(-2)) > 0$, then $u(\vec{e}, \vec{f}) \geq \frac{g+3}{4} - 4$.
\end{proof}

\begin{rem}
Aaron Landesman points out that our above Lemma \ref{coh-4} parallels \cite[Lemma 11]{BD4} of Bhargava. Bhargava's two cases $a_{11} = 0$ or $a_{11} = a_{12} = 0$ correspond to the fact that either $\O(2e_1 - f_2)$ or $\O(2e_1 - f_2)$ and $\O(e_1 + e_2 - f_2)$ are the only possible negative summands of $F_\alpha^\vee \otimes \Sym^2 E_\alpha$ for a non-factoring cover $\alpha$. %We believe the coefficient of $g$ here to be sharp in our setting. Heuristically, the strength of our bound seems to be between Bhargava's estimate and an unpublished improvement of Tsimerman and Shankar (which should apply only in the number field setting).
\end{rem}

Lemmas \ref{nfshave} and \ref{coh-4} together should be thought of as saying that $\H_{4,g}^{\circ}$ and $\H_{4,g}^{\nf}$ are ``good approximations" of each other. We can now complete the proof of Theorem \ref{thm4}.

\begin{proof}[Proof of Theorem \ref{thm4}]
Suppose $i < \frac{g+3}{4} - 4$. Consider the restriction maps
\[A^i(\H_{4,g}^{\mathrm{nf}}) \xrightarrow{\sim} A^i(\H_{4,g}^{\mathrm{nf}} \cap \H_{4,g}^\circ) \xleftarrow{\sim} A^i(\H_{4,g}^\circ).\]
Lemma \ref{coh-4} says the arrow on the left is an isomorphism; Lemma \ref{nfshave} says the arrow on the right is an isomorphism.
In turn then, we also have
\[R^i(\H_{4,g}^{\mathrm{nf}}) \xrightarrow{\sim} R^i(\H_{4,g}^{\mathrm{nf}} \cap \H_{4,g}^\circ) \xleftarrow{\sim} R^i(\H_{4,g}^\circ),\]
where $R^i$ of an open substack of $\H_{4,g}$ means the image of tautological classes under the restriction to that open.
Since $\H_{4,g}^\circ \subset \H_{4,g}'$, Lemma \ref{Hprime4} implies $A^i(\H_{4,g}^\circ) = R^i(\H_{4,g}^\circ)$. Thus, we conclude \[A^i(\H_{4,g}^{\mathrm{nf}}) = A^i(\H_{4,g}^\circ) = R^i(\H_{4,g}^\circ) = R^i(\H_{4,g}^{\mathrm{nf}}). \qedhere\]
\end{proof}

\subsection{Degree $5$} \label{op5sec}
Using Casnati's characterization of regular degree $5$ covers (Theorem \ref{C5}), a regular degree $5$ cover of is equivalent to the data of a rank $4$ vector bundle $E$; a rank $5$ vector bundle $F$; an isomorphism $(\det E)^{\otimes 2}\cong \det F$; and a global section of $\H om(E^{\vee}\otimes \det E,\wedge^2 F)$ satisfying certain conditions.
By Example \ref{p1ex}, if a cover $\alpha: C \to \pp^1$ has genus $g$, then $\deg(E_\alpha) = g+4$. In turn, $\deg(F_\alpha) = 2\deg(E_\alpha) = 2g + 8$.
To build the appropriate base stack, we start with $\Ub_{4,g+4}\times_
{\BSL_2}\Ub_{5,2g+8}$ which parametrizes tuples $(S, V, E, F)$ where $V$ is a rank $2$ vector bundle on $S$ with trivial determinant, and $E$ and $F$ are vector bundles of the appropriate ranks and degrees on $\pp V$. We let $\E$ denote the universal rank $4$ vector bundle and $\F$ the universal rank $5$ bundle on  the universal $\p^1$-bundle $\pi: \P \to \Ub_{4,g+4}\times_
{\BSL_2}\Ub_{5,2g+8}$. 
Since $\det \E^{\otimes 2} \otimes \det \F^\vee$ is a line bundle of degree $0$ on each fiber of $\pi$, we have $\det \E^{\otimes 2} \otimes \det \F^\vee \cong \pi^* \mathcal{L}$ where
$\mathcal{L}:=\pi_*(\det \E^{\otimes 2}\otimes \det \F^{\vee}),$
which is a line bundle by cohomology and base change. 
\begin{definition}
With notation as above, we define the stack $\B_{5,g}$ as the $\mathbb{G}_m$-torsor over $\Ub_{5,g+4}\times_{\BSL_2}\Ub_{5,2g+8}$ given by the complement of the zero section of the line bundle $\mathcal{L}$.
\end{definition}

By slight abuse of notation, we continue to denote the universal $\pi: \pp^1$-bundle by $\P \to \B_{5,g}$ and the universal rank $4$ and $5$ vector bundles on it by $\E$ and $\F$.
Let $z = \O_{\P}(1)$ and write
\[c_i(\E) = \pi^*a_i + (\pi^*a_i') z \qquad \text{and} \qquad c_i(\F) = \pi^*b_i + (\pi^*b_i') z. \]
for $a_i, b_i \in A^i(\B_{5,g})$ and $a_i', b_i' \in A^{i-1}(\B_{5,g})$.
Note that $2a_1' = b_1' = 2(g+4)$.
Moreover, by definition of $\B_{5,g}$, we have $c_1(\det \E^{\otimes 2} \otimes \det \F^\vee) = 0$, so $b_1 = 2a_1$.
Using Lemma \ref{Gmbundles} and Theorem \ref{Vchow} as in the previous subsection, we have
\begin{equation} \label{b5}\trun^{g+5} A^*(\B_{5,g}) = \trun^{g+5} \qq[c_2, a_1, \ldots, a_4, a_2', \ldots, a_4', b_2, \ldots, b_5, b_2', \ldots, b_5']. 
\end{equation}

We define $\aV_{5,g} := \H om(\E^{\vee}\otimes \det \E,\wedge^2 \F)$, and $\B_{5,g}'$ and $\B_{5,g}^\circ$ as in \eqref{Bprime} and \eqref{Bcirc}, respectively.
Given a map $S \rightarrow \H_{5,g}$, let $\pi_S: \P_S \rightarrow S$ denote the $\pp^1$-bundle and let $\E_S$ (resp. $\F_S)$ be the rank $4$ (resp. rank $5$) vector bundle on $\P_S$ associated to the family in the sense of Casnati--Ekedahl.
The open substacks $\H^\circ_{5,g} \subseteq \H'_{5,g} \subseteq \H_{5,g}$ are defined by
\begin{align*}
\{S \rightarrow \H^\circ_{5,g}\} &= \{S \rightarrow \H_{5,g} : R^1(\pi_S)_*( 
\H om(\E^\vee_S \otimes \det \E_S, \wedge^2 \F_S) \otimes \O_{\P_S}(-2)) = 0 \\
& \qquad \qquad \qquad \qquad \text{and $\F_S$ globally generated on fibers of $\pi_S$}\}. \\
\{S \rightarrow \H'_{5,g}\} &= \{S \rightarrow \H_{5,g} : R^1(\pi_S)_*( 
\H om(\E^\vee_S \otimes \det \E_S, \wedge^2 \F_S)) = 0 \\
& \qquad \qquad \qquad \qquad \text{and $\F_S$ globally generated on fibers of $\pi_S$}\}.
\end{align*}

The important feature of the open $\H_{5,g}'$ is that it can be realized as an open inside the vector bundle $\aW_{5,g}' := \pi_* \aV_{5,g}|_{\B_{5,g}'}$ over $\B_{5,g}'$.
\begin{lem}\label{Hprime5}
There is an open inclusion $\H_{5,g}' \to \aW_{5,g}'$. In particular, $A^*(\H_{5,g}') = R^*(\H_{5,g}')$ is generated by the CE classes $c_2, a_1, \ldots, a_4, a_2', \ldots, a_4', b_2, \ldots, b_5, b_2', \ldots, b_5'$.% Repeating all constructions with script letters, we obtain an open inclusion $\Hp'_{5,g} \to \Wp_{5,g}$.
\end{lem}
\begin{proof}
The objects of $\aW_{5,g}'$ are tuples $(S,V, E, F, \phi, \eta)$ where  $(S,V, E, F, \phi) \in \B_{5,g}'$ and $\eta \in H^0(\pp V, \H om(E^\vee \otimes \det E, \wedge^2 F))$. Using the notation of Section \ref{pentsec}, we define $\mathcal{Y}_{5,g}' \subset \aW_{5,g}'$ to be the open substack defined by the condition that $D(\Phi(\eta)) \subset \pp E^\vee \to S$ is a family of smooth curves. Considering their Hilbert polynomials as determined by the resolution \eqref{res5}, we see that the fibers of $D(\Phi(\eta)) \to S$ have arithmetic genus $g$.
Applying Theorem \ref{C5}, we see that $\H_{5,g}'$ is equivalent to $\mathcal{Y}_{5,g}'$

%If we wish to work with $\pp^1$-fibrations, we replace all caligraphic letters with the same script letters and $\O_{\P}(-2)$ with $\omega_{\pi}$ in defining $\Hp_{5,g}'$.

By excision, the Chow ring of $\H_{5,g}'$ is generated by restriction of classes from $\aW_{5,g}'$. Since $\aW_{5,g}'$ is a vector bundle over $\B_{5,g}'$, their Chow rings are isomorphic, so the statement about generators follows from Theorem \ref{Vchow}.
\end{proof}

Now we show that the complements of the opens we have defined have high codimension.
\begin{lem} \label{b5circ}
The support of $R^1\pi_*(\aV_{5,g} \otimes \O_{\P}(-2))$ has codimension at least $\frac{g+4}{5} - 16$. That is, the codimension of the complement of $\B_{5,g}^\circ \subset \B_{5,g}$ is at least $\frac{g+4}{5} - 16$.
%In particular,
%\[A^*(\B_{5,g}^\circ)|_{t_5(g)} = \qq[c_2, a_1, \ldots, a_4, a_2', \ldots, a_4', b_2, \ldots, b_5, b_2', \ldots, b_5']|_{t_5(g)}.\]
\end{lem}
\begin{proof}
By \eqref{simsp}, the codimension of the support of $R^1\pi_*(\H om(\E^{\vee}\otimes \det \E,\wedge^2 \F) \otimes \O_{\P}(-2))$ is the minimum value of
\[h^1(\pp^1, \E nd(\O(\vec{e}))) + h^1(\pp^1, \E nd(\O(\vec{f})))\]
as we range over splitting types $\vec{e}$ of degree $g +4$ and $\vec{f}$ of degree $2g+8$ so that 
\[ h^1(\pp^1, \H om(\O(\vec{e})^\vee \otimes \det \O(\vec{e}), \wedge^2 \O(\vec{f})) \otimes \O_{\pp^1}(-2)) > 0 \quad \Leftrightarrow \quad e_1 + f_1 + f_2 - (g+4) \leq 0.\]

Similar to the proof of Lemma \ref{b4}, we may find this minimum by finding the minimum of the function
\[f(x_1,\ldots, x_4, y_1, \ldots, y_5) = 3x_4 + x_3 - x_2 - 3x_1 + 4y_5 + 2y_4 - 2y_2 - 4y_1\]
on the compact region $D$ defined by 
\begin{gather*}
0 \leq x_1 \leq \cdots \leq x_4, \quad x_1 + \ldots + x_4 = 1, \quad 0 \leq y_1 \leq \cdots \leq y_5, \quad y_1 + \ldots + y_5 = 2 \\
x_1 + y_1 + y_2 - 1 \leq 0.
\end{gather*}
Using our code \cite{github},
we find that the minimum of the linear function $f$ over $D$ is $\frac{1}{5}$ attained at
$(\frac{1}{5},\frac{4}{15},\frac{4}{15},\frac{4}{15},\frac{2}{5},\frac{2}{5},\frac{2}{5},\frac{2}{5},\frac{2}{5})$.
Thus,
\begin{align*}
\dim \Supp R^1 \pi_*(\aV_{5,g} \otimes \O_{\P}(-2)) &\geq (g+4) \cdot \min_D(f) - 16 = \frac{g+4}{5} - 16. \qedhere
\end{align*}
\end{proof}

For later use, let us note an immediate consequence of the previous lemma: Using excision and \eqref{b5}, we see
\begin{equation} \label{b5forlater}
\trun^{(g+4)/5 - 16} A^*(\B_{5,g}^\circ) = \trun^{(g+4)/5 - 16} \qq[c_2, a_1, \ldots, a_4, a_2', \ldots, a_4', b_2, \ldots, b_5, b_2', \ldots, b_5'].
\end{equation}

\begin{lem} \label{coh-5}
The codimension of the locus of smooth degree $5$ covers $\alpha$ such that \[h^1(\H om(E_\alpha^\vee \otimes \det E_\alpha, \wedge^2 F_\alpha) \otimes \O_{\pp^1}(-2)) > 0\]
has codimension at least $\frac{g+4}{5} - 16$. That is, the codimension of the complement of $\H^\circ_{5,g}$ inside $\H_{5,g}$ is at least $\frac{g+4}{5} - 16$.
\end{lem}
\begin{proof}
The cohomological statement depends only on the splitting type of $E_\alpha$ and $F_\alpha$. 
In the proof of \cite[Proposition 5.2]{DP}, Deopurkar--Patel show that  
the codimension of the locus of covers such that $E_\alpha \cong \O(\vec{e})$ and $F_\alpha \cong \O(\vec{f})$ has codimension
\begin{align} \label{pentcodim}
u(\vec{e}, \vec{f}) &:= h^1(\pp^1, \E nd(\O(\vec{e}))) + h^1(\pp^1, \E nd(\O(\vec{f}))) \\
&\qquad - h^1(\pp^1, \O(\vec{e}) \otimes \wedge^2 \O(\vec{f}) \otimes \O_{\pp^1}(-g-4)). \notag
\end{align}

A cover with these discrete invariants corresponds to a global section $\eta$ of 
\[\H om(\O(\vec{e})^\vee \otimes \det \O(\vec{e}), \wedge^2 \O(\vec{f})) = \O(\vec{e}) \otimes \wedge^2 \O(\vec{f}) \otimes \O_{\pp^1}(-g-4).\]
Such a global section can be represented by a skew-symmetric matrix
\begin{equation} \label{mat} M_\eta = \left(\begin{matrix} 0 & L_{1,2} & L_{1,3} & L_{1,4} & L_{1,5} \\ -L_{1,2} & 0 & L_{2,3} & L_{2,4} & L_{2,5} \\ -L_{1,3} & -L_{2,3} & 0 & L_{3,4} & L_{3,5} \\ -L_{1,4} & -L_{2,4} & -L_{3,4} & 0 & L_{4,5} \\ -L_{1,5} & -L_{2,5} & -L_{3,5} & -L_{4,5} & 0 \end{matrix} \right), \end{equation}
where $L_{i,j} \in H^0(\O(f_i + f_j) \otimes \O(\vec{e}) \otimes \O(-g-4))$.
The corresponding curve $C \subset \pp E^\vee$ is cut out by the $4 \times 4$ Pfaffians of the main minors of $M_\eta$. The Pfaffian of the submatrix obtained by deleting the last row and column is
\[Q_5 = L_{1,2} L_{3,4} - L_{1,3} L_{2,4} + L_{2,3}L_{1,4}.\]
If $Q_5$ is reducible, then $C$ is reducible. Indeed, if $C$ were irreducible, it would be contained in one component of $Q_5$, forcing every fiber to be contained in a hyperplane, violating the Geometric-Riemann-Roch theorem. 
Therefore, as observed in \cite[p. 21]{DP}, $L_{1,2}$ and $L_{1,3}$ cannot both be identically zero, and so
\begin{equation} \label{lower}
    f_1 + f_3 + e_4 - (g + 4) \geq 0.
\end{equation}Let $X_1, \ldots, X_4$ be coordinates on $\pp E^\vee$ corresponding to a choice of splitting $E \cong \O(\vec{e})$, so we think of $L_{i,j}$ as a linear form in $X_1, \ldots, X_4$ where the coefficient of $X_k$ is a section of $\O(f_i + f_j) \otimes \O(e_k) \otimes \O(-g-4)$, i.e. a homogeneous polynomial of degree $f_i + f_j + e_k - (g+4)$ on $\pp^1$.
If $Q_5$ is irreducible, it cannot be divisible by $X_4$.
Observe that if $f_i + f_j + e_3 - (g+4) < 0$, then the coefficients of $X_k$ for $k \leq 3$ vanish, so $X_4$ divides $L_{i,j}$. If $X_4$ divides $L_{1,2}$, $L_{1,3}$ and $L_{1,4}$, then $X_4$ divides $Q_5$ and $Q_5$ is reducible. To prevent this, we must have
\begin{equation} \label{imp2}
f_1 + f_4 + e_3 - (g+4) \geq 0.
\end{equation}
Similarly, if $X_4$ divides $L_{1,2}$, $L_{1,3}$ and $L_{2,3}$, then $X_4$ divides $Q_5$ and $Q_5$ is reducible. To prevent this, we must have
\begin{equation} \label{imp3}
f_2 + f_3 + e_3 - (g + 4) \geq 0.
\end{equation}

For splitting types satisfying \eqref{lower}, \eqref{imp2}, and \eqref{imp3}, at most $11$ of the $40$ summands of the form $\O(e_i + f_j + f_k - (g+4))$ in $\O(\vec{e}) \otimes \wedge^2 \O(\vec{f}) \otimes \O_{\pp^1}(-g-4)$ can be negative. For these allowed splitting types, we have
\begin{align*}
u(\vec{e}, \vec{f}) &= h^1(\pp^1, \E nd(\O(\vec{e}))) + h^1(\pp^1, \E nd(\O(\vec{f}))) \\
&\qquad - \sum_{i=1}^4 \max\{0, g+3 - f_1 - f_2 - e_i\} -\sum_{i=1}^3 \max\{0, g + 3 - f_1 - f_3 - e_i\} \\
&\qquad - \sum_{i = 1}^2 \max\{0, g+3 - f_1 - f_4 - e_i\} - \sum_{i = 1}^2 \max\{0, g+3 - f_2 - f_3 - e_i\}.
\end{align*}
We seek a lower bound on $u(\vec{e}, \vec{f})$ given that $\O(\vec{e}) \otimes \wedge^2 \O(\vec{f}) \otimes \O(-g-4)$ has a non-positive summand, i.e. in the region where $e_1 + f_1 + f_2 - (g+4) \leq 0$.
Note that
\begin{align*}
h^1(\pp^1, \E nd(\O(\vec{e})) &\geq 3e_4 + e_3 - e_2 - 3e_1 - 6  \\
h^1(\pp^1, \E nd(\O(\vec{f})) &\geq 4f_5 + 2f_4 - 2f_2 - 4f_1 - 10.
\end{align*}

Let us define a function of $9$ real variables
\begin{align*}
f(x_1,\ldots, x_4, y_1, \ldots, y_5) &:= 3x_4 + x_3 - x_2 - 3x_1 + 4y_5 + 2y_4 - 2y_2 - 4y_1 \\
&\quad - \sum_{i = 1}^4 \max\{0, 1 - y_1 - y_2 - x_i\}  -\sum_{i=1}^3 \max\{0, 1 - y_1 - y_3 - x_i\} \\
&\quad - \sum_{i = 1}^2 \max\{0, 1 - y_1 - y_4 - x_i\} - \sum_{i = 1}^2 \max\{0, 1 - y_2 - y_3 - x_i\}
\end{align*}
so that
\[u(\vec{e}, \vec{f}) \geq (g+4) f\left(\frac{e_1}{g+4}, \ldots, \frac{e_4}{g+4}, \frac{f_1}{g+4}, \ldots, \frac{f_5}{g+4}\right)
- 16.
\]
Now we wish to find the minimum of $f$ on the compact region defined by
\begin{gather*}
0 \leq x_1 \leq \cdots \leq x_4, \quad x_1 + \ldots + x_4 = 1, \quad 0 \leq y_1 \leq \cdots \leq y_5, \quad y_1 + \ldots + y_5 = 2 \\
 y_1 + y_3 + x_4 - 1 \geq 0, \quad
y_1 + y_4 + x_3 - 1 \geq 0, \qquad
y_2 + y_3 + x_3 -1 \geq 0 \\
x_1 + y_1 + y_2 - 1 \leq 0.
\end{gather*}
Since $f$ is piecewise linear, its extreme values are attained at points where multiple boundary conditions (including those where the linear function changes) intersect to give a single point. Our code \cite{github} performs the linear algebra to locate such points and determines that the minimum is $\frac{1}{5}$, which is attained at $(\frac{1}{5}, \frac{4}{15}, \frac{4}{15}, \frac{4}{15}, \frac{2}{5}, \frac{2}{5}, \frac{2}{5}, \frac{2}{5}, \frac{2}{5})$.
It follows that if $\vec{e}$ and $\vec{f}$ satisfy $h^1(\pp^1, \O(\vec{e}) \otimes \wedge^2\O(\vec{f}) \otimes \O(-g-4) \otimes \O(-2)) > 0$ then $u(\vec{e}, \vec{f}) \geq \frac{g + 4}{5} - 16$.
\end{proof}

We now prove Theorem \ref{thm5} to complete the $k = 5$ case.
\begin{proof}[Proof of Theorem \ref{thm5}]
Suppose $i < \frac{g+4}{5} - 16$. Then, by excision and Lemma \ref{coh-5}, the restriction map $A^i(\H_{5,g}) \to A^i(\H_{5,g}^\circ)$ is an isomorphism.
Hence, $R^i(\H_{5,g}) \to R^i(\H_{5,g}^\circ)$ is also an isomorphism. Since $\H_{5,g}^\circ \subseteq \H_{5,g}'$, Lemma \ref{Hprime5} tells us that $A^i(\H_{5,g}^\circ) = R^i(\H_{5,g}^\circ)$. Hence, we have shown
\[A^i(\H_{5,g}) = A^i(\H_{5,g}^\circ) =  R^i(\H_{5,g}^\circ) = R^i(\H_{5,g}). \qedhere\]
\end{proof}

\section{Conclusion and preview of subsequent work}
At this point, we have established that, for $k \leq 5$, a large portion of the Chow ring of the (non-factoring) Hurwitz space $\H_{k,g}^{\mathrm{nf}}$ is tautological. We did so by showing that $\H_{k,g}^{\nf}$ is closely approximated by an open substack $\H_{k,g}^\circ$ which, in turn, can be realized as an open substack of a vector bundle $\aW_{r,d}^\circ = \aW_{r,d}'|_{\B_{r,d}^\circ}$ over the stack $\B_{r,d}^\circ$. By \eqref{b4forlater} and \eqref{b5forlater}, we understand $A^*(\B_{r,d}^\circ) \cong A^*(\aW_{r,d}^\circ)$ well. First of all, we know $A^*(\B_{r,d}^\circ)$ is generated by classes which pullback to the CE classes on $\H_{k,g}^\circ$; this is how we saw $A^*(\H_{k,g}^\circ) = R^*(\H_{k,g}^\circ)$. However, we actually know a bit more: the generators we list for $A^*(\B_{r,d}^\circ) \cong A^*(\aW_{r,d}^\circ)$ satisfy no relations in low degrees. In other words, all relations that their pullbacks to $\H_{k,g}^\circ$ satisfy come from performing excision on the complement of $\H_{k,g}^\circ \subset \aW_{r,d}^\circ$.

Determining these relations will be the focus of our subsequent work \cite{part2}. The central innovation there is to find an appropriate resolution of the complement of $\H_{k,g}^\circ \subset \aW_{r,d}^\circ$, which allows us to determine \emph{all} relations in degrees up to roughly $g/k$. Furthermore, we will prove that the relations we find among the restrictions of CE classes to $\H_{k,g}^\circ$ actually hold on \emph{all} of $\H_{k,g}$. Using the codimension bounds we established in Section 5 here, results about the structure of $A^*(\H_{k,g}^\circ) = R^*(\H_{k,g}^\circ)$ will then translate into results about the structure of $A^*(\H_{k,g}^{\nf})$ and $R^*(\H_{k,g})$ in degrees up to roughly $g/k$.

\bibliographystyle{amsplain}
\bibliography{refs}
\end{document}